\documentclass[11pt]{article}
\usepackage[all]{xy}
\usepackage{amsfonts}
\usepackage{amsthm}
\usepackage{enumerate}
\usepackage{graphicx}
\usepackage{mathrsfs}
\usepackage{bm}
\usepackage{cite}
\usepackage{amssymb,amsmath} 
\usepackage{makecell}
\usepackage{tikz}
\usepackage{pgf}
\usepackage{tikz}
\usetikzlibrary{patterns}
\usepackage{pgffor}
\usepackage{pgfcalendar}
\usepackage{pgfpages}
\usepackage{shuffle,yfonts}
\usepackage{mathtools}
\DeclareFontFamily{U}{shuffle}{}
\DeclareFontShape{U}{shuffle}{m}{n}{ <-8>shuffle7 <8->shuffle10}{}

\newcommand{\bfs}{{\boldsymbol{\sl{s}}}}

 \allowdisplaybreaks
\usetikzlibrary{arrows,shapes,chains}
 \allowdisplaybreaks

\catcode`!=11
\let\!int\int \def\int{\displaystyle\!int}
\let\!lim\lim \def\lim{\displaystyle\!lim}
\let\!sum\sum \def\sum{\displaystyle\!sum}
\let\!sup\sup \def\sup{\displaystyle\!sup}
\let\!inf\inf \def\inf{\displaystyle\!inf}
\let\!cap\cap \def\cap{\displaystyle\!cap}
\let\!max\max \def\max{\displaystyle\!max}
\let\!min\min \def\min{\displaystyle\!min}
\let\!frac\frac \def\frac{\displaystyle\!frac}
\catcode`!=12

\let\oldsection\section
\renewcommand\section{\setcounter{equation}{0}\oldsection}

\allowdisplaybreaks

\def\R{\mathbb{R}}

\def\N{\mathbb{N}}
\def\Z{\mathbb{Z}}
\def\Q{\mathbb{Q}}

\def\ze{\zeta}

\theoremstyle{plain}
\newtheorem{thm}{Theorem}[section]
\newtheorem{lem}[thm]{Lemma}
\newtheorem{cor}[thm]{Corollary}
\newtheorem{con}[thm]{Conjecture}
\newtheorem{pro}[thm]{Proposition}
\theoremstyle{definition}
\newtheorem{defn}{Definition}[section]
\newtheorem{re}[thm]{Remark}
\newtheorem{exa}[thm]{Example}

\begin{document}
\title{\bf Mixed Berndt-Type Integrals and Generalized Barnes Multiple Zeta Functions}
\author{
		{ Jianing Zhou\thanks{Email: 202421511272@smail.xtu.edu.cn}}\\[1mm]
		\small School of
		Mathematics and Computational Science, Xiangtan University, \\ \small Xiangtan 411105,  P.R. China}
	
\date{}
\maketitle
	
\noindent{\bf Abstract.} In this paper, we define and study four families of Berndt-type integrals, called mixed Berndt-type integrals, which contain (hyperbolic) sine and cosine functions in the integrand function.  Using contour integration, these integrals are first converted to some hyperbolic (infinite) sums of Ramanujan type, all of which can be calculated in closed form by  comparing both the Fourier series expansions and the Maclaurin series expansions of certain Jacobi elliptic functions. These sums can be expressed as rational polynomials of $\Gamma(1/4)$ and $\pi^{-1}$ which give rise to the closed formulas of the mixed Berndt-type integrals we are interested in. Moreover, we also present some interesting consequences and illustrative examples. Additionally, we define a generalized Barnes multiple zeta function, and find a classic integral representation of the generalized Barnes multiple zeta function. Furthermore, we give an alternative evaluation of the mixed Berndt-type integrals in terms of the generalized Barnes multiple zeta function. Finally, we obtain some direct evaluations of rational linear combinations of the generalized Barnes multiple zeta function.


	\tableofcontents
\section{Introduction}

Let $m>0$ be an integer. The \emph{Berndt-type integrals of order $p$} are defined by the form (see Bradshaw-Vignat \cite{BV2024} and Xu-Zhao \cite{XZ2024})
\begin{align}\label{BTI-definition-1}
I_{\pm }(s,m):=\int_0^\infty \frac{x^{s-1}\mathrm{d}x}{(\cos x\pm\cosh x)^m},
\end{align}
where $m\geq 1$ and $s\geq 1$ if the denominator has ``$+$' sign, and $s\geq 2m+1$ otherwise. This improper integral traces back to S. Ramanujan, who first submitted it as a problem over a century ago to the \emph{Indian Journal of Pure and Applied Mathematics} \cite[pp. 325-326]{Rama1916}:
\begin{equation*}
\int_0^\infty \frac{\sin(nx)}{x(\cos x+\cosh x)}\mathrm{d}x=\frac{\pi}{4}.
\end{equation*}
Here, $n$ is an odd positive integer.  After submission, Wilkinson \cite{W1916} proved it.

 About 80 years later, Ismail and Valent \cite{Ismail1998} found the mystery integral (\emph{Ismail-type integrals})
\begin{align}\label{A.KuznetsovInt}
	\int_{-\infty}^{\infty}{\frac{\mathrm{d}t}{\cos \left( K\sqrt{t} \right) +\cosh \left( K'\sqrt{t} \right)}=2},
\end{align}
	where $K
	\equiv K(x)$ denotes the complete elliptic integral of the first kind and $K'=K(1-x)$, see Section 2. The integrand in \eqref{A.KuznetsovInt} is the weight function for a sequence of polynomials arising in the Hamburger moment problem studied by Ismail and Valent. The systematic study of Berndt-type integrals began with the works of Kuznetsov \cite{K2017} and Berndt \cite{Berndt2016} about ten years ago. Some recent results on infinite series involving Berndt-type integrals may be found in the works of \cite{BV2024,PanWang2025,RXYZ2024,RXZ2023,XZ2023,XZ2024,ZhangRui2024} and the references therein. Kuznetsov \cite{K2017} used the contour integration and theta functions to prove a more general result of Ismail-type
integrals
\begin{align}\label{Ims-Kuz+}
\frac1{2}\int_{-\infty}^\infty \frac{t^n \mathrm{d}t}{\cos(K\sqrt{t})+\cosh(K'\sqrt{t})}=(-1)^n \frac{\mathrm{d}^{2n+1}}{\mathrm{d}u^{2n+1}} \frac{{\rm sn} (u,k)}{{\rm cd}(u,k)}|_{u=0},
\end{align}
where $x=k^2\ (0<k<1)$ and the Jacobi elliptic function ${\rm sn} (u,k)$ is defined via the inversion of the elliptic integral
\begin{align}
u=\int_0^\varphi \frac{\mathrm{d}t}{\sqrt{1-k^2\sin^2 t}}\quad (0<k^2<1),
\end{align}
namely, ${\rm sn}(u):=\sin \varphi$. As before, we refer $k$ as the elliptic modulus. We also write $\varphi={\rm am}(u,k)={\rm am}(u)$ and call it the Jacobi amplitude.
The Jacobi elliptic function ${\rm cd} (u,k)$ may be defined as follows:
\begin{align*}
& {\rm cd}(u,k):=\frac{{\rm cn}(u,k)}{{\rm dn}(u,k)},
\end{align*}
where ${\rm cn}(u,k):=\sqrt{1-{\rm sn}^2 (u,k)}$ and ${\rm dn} (u,k):=\sqrt{1-k^2{\rm sn}^2(u,k)}$. Berndt \cite{Berndt2016} provided a direct evaluation of \eqref{A.KuznetsovInt} and other similar integrals of this type by using Cauchy's residue theory, the Fourier series expansions, and the Maclaurin series expansions of certain Jacobi elliptic functions. In particular, using Berndt's results \cite{Berndt2016} and Xu-Zhao's result \cite{XZ2024}, it is easy to see that the Berndt-type integrals of order one have the following structures
\begin{align*}
I_{-}(4p,1)\in \Q \frac{\Gamma^{8p}(1/4)}{\pi^{2p}}\quad (p\in \N)\quad\text{and}\quad I_{+}(4p+2,1)\in \Q \frac{\Gamma^{8p+4}(1/4)}{\pi^{2p+1}}\quad (p\in \N_0:=\N\cup\{0\}).
\end{align*}
Moreover, Berndt \cite[Thm. 6.1]{Berndt2016} also considered some integrals containing $\sinh x\pm i \sin x$ in their integrands instead of $\cosh x\pm \cos x$.
Further, using the contour integrals and the Fourier series expansion and Maclaurin series expansion of a certain Jacobi elliptic function, Rui-Xu-Yang-Zhao obtained a structural result on the original Berndt type integral involving (hyperbolic) sine functions (see \cite[Thm. 3.4]{RXYZ2024})
\begin{align}\label{equ-org-BerndtInte}
\int_0^\infty \frac{x^{2m} (\sinh x-(-1)^m \sin x)}{\sinh^2 x+\sin^2 x}\mathrm{d}x\in \Q \frac{\Gamma^{4m+2}(1/4)}{\pi^{(2m+1)/2}}\sqrt{2}\quad (m\in \N)
\end{align}
and prove the following result involving another family of Berndt-type integrals (see \cite[Cor. 6.5]{RXYZ2024}):
\begin{align} \label{defn:BerdtIntegral}
\int_{0}^{\infty}\frac{x^{4p+1}(\sin^2 x-\sinh^2 x)}{\left(\sinh^2 x+\sin^2 x\right)^2}\mathrm{d}x\in \Q\frac{\Gamma^{8p+4}(1/4)}{\pi^{2p+2}}+\Q\frac{\Gamma^{8p+8}(1/4)}{\pi^{2p+4}}\quad (p\in \N).
\end{align}
Recently, by extending an argument as used in the proof of the main theorems of \cite{Berndt2016}, Xu and Zhao \cite{XZ2023} derived explicit evaluations for two families of order-two Berndt-type integrals (cf. \cite[Thm. 1.3]{XZ2023}):
\begin{align*}
&I_{\pm}(4p+2,2)\in \Q\frac{\Gamma^{8p}(1/4)}{\pi^{2p}}+\Q\frac{\Gamma^{8p+8}(1/4)}{\pi^{2p+4}}\quad (p\in \N).
\end{align*}
Further, Rui, Xu and Zhao \cite[Thm. 1.1]{RXZ2023} established explicit evaluations for two families of order-three Berndt-type integrals using special values of the Gamma function.
More generally, let $X=\Gamma^4(1/4)$ and $Y=\pi^{-1}$, and $\Q[x, y]$ denotes all rational linear combinations of products of $x$ and $y$ (the
polynomial ring generated by $x$ and $y$ with rational coefficients). Xu and Zhao \cite{XZ2023,XZ2024} proved two structural theorems on the general Berndt-type integrals with the denominator having arbitrary positive degrees. Specifically, they have shown that for all integers $m\geq 1$ and $p\geq [m/2]$, the Berndt-type integrals satisfy
\begin{align*}
I_{+}(4p+2,m)\in\mathbb{Q}[X,Y],
\end{align*}
where the degrees of $X$ have the same parity as that of $m$ and are between $2p-m+2$ and $2p+m$, inclusive, while the degrees of $Y$ are between $2p-m+2$ and $2p+3m-2$, inclusive. Similarly, they have shown that
\begin{align*}
I_{-}(4p,2m+1)\in\mathbb{Q}[X,Y]
\end{align*}
for all $p\geq m+1\geq 1$ and that the degrees of $X$ in each term are always even and between $2p-2m$ and $2p+2m$, while the degrees of $Y$ are between $2p-2m$ and
$2p+6m$. In the even order case, they have shown that
\begin{align*}
I_{-}(4p+1,2m)\in\mathbb{Q}[X,Y]
\end{align*}
for all integers $p\geq m \geq 1$, where the powers of $X$  lie between $2p+2-2m$ and $2p+2m$, while the degrees of $Y$ are between $2p+2-2m$ and $2p+6m-2$. Very recently,  Bradshaw and Vignat  \cite[Cor. 2]{BV2024} surprisingly discovered a close connection between Berndt-type integrals and Barnes zeta functions. Moreover, Bradshaw and Vignat \cite[Cor. 5]{BV2024} used the Kuznetsov's method to prove the following relation between Jacobi elliptic functions and another Ismail-type integrals: for $n\in \N_0$,
\begin{align}\label{Ims-Kuz-1}
\frac1{2}\int_{-\infty}^\infty \frac{x^{n+1} dx}{\cos(\sqrt{x}K)-\cosh(\sqrt{x}K')}=(-1)^{n+1}4 \frac{d^{2n+1}}{du^{2n+1}} \frac{{\rm sn}^2 (u,k)}{{\rm cd}^2(u,k){\rm sd}(2u,k)}|_{u=0},
\end{align}
where the Jacobi elliptic function ${\rm sd}(u)\equiv {\rm sd}(u,k)$  defined by
\begin{align}\label{defin-sd}
	{\rm sd}(u,k):=\frac{{\rm sn}(u,k)}{{\rm dn}(u,k)}.
\end{align}

From the above, it can be seen that the integrands in the Berndt-type integrals studied above either only contain (hyperbolic) sine functions or (hyperbolic) cosine functions.
In this paper, we consider the following infinite integrals containing (hyperbolic) sine functions and (hyperbolic) cosine functions in their integrands:
\begin{align}
B^{\pm }_{\pm}(s):=\int_0^{\infty}{\frac{x^s\left( \sinh x\pm \sin x \right) \mathrm{d}x}{\left(\sinh ^2x+\sin ^2x\right) \left( \cosh x\pm \cos x\right)}}.
\end{align}
We call them \emph{mixed Berndt-type integrals}. Using contour integration and series associated with Jacobi elliptic functions, we establish two structural theorems for these integrals. Specifically,  we prove the following evaluations: let $\Gamma:=\Gamma(1/4)$, for $m\in \N$,
\begin{align}
B^{- }_{+}(4m)=\int_0^{\infty}{\frac{x^{4m}\left( \sinh x-\sin x \right) \mathrm{d}x}{\left( \sinh ^2x+\sin ^2x \right)\left( \cosh x+\cos x \right)}}&\in \mathbb{Q} \frac{\Gamma ^{8m}}{\pi ^{2m}}+\mathbb{Q} \frac{\Gamma ^{8m+4}}{\pi ^{2m+2}},
\\
B^{+}_{+}(4m-2)=\int_0^{\infty}{\frac{x^{4m-2}\left( \sin x+\sinh x \right) \mathrm{d}x}{\left(\sinh ^2x+\sin ^2x \right) \left(\cosh x+\cos x \right)}}&\in \mathbb{Q} \frac{\Gamma ^{8m-4}}{\pi ^{2m-1}}+\mathbb{Q} \frac{\Gamma ^{8m}}{\pi ^{2m+1}},
  \end{align}
  and
\begin{align}
B^{+}_{-}(4m)=\int_0^{\infty}{\frac{x^{4m}\left( \sinh x+\sin x \right) \mathrm{d}x}{\left(\sinh ^2x+\sin ^2x \right) \left( \cosh x-\cos x \right)}}&\in \mathbb{Q} \frac{\Gamma ^{8m}}{\pi ^{2m}}+\mathbb{Q} \frac{\Gamma ^{8m+4}}{\pi ^{2m+2}},
\\
B^{-}_{-}(4m-2)=\int_0^{\infty}{\frac{x^{4m-2}\left( \sinh x-\sin x \right) \mathrm{d}x}{\left(\sinh ^2x+\sin ^2x \right) \left(\cosh x-\cos x \right)}}&\in \mathbb{Q} \frac{\Gamma ^{8m-4}}{\pi ^{2m}}+\mathbb{Q} \frac{\Gamma ^{8m}}{\pi ^{2m+1}}\quad (m\ge 2).
  \end{align}
Furthermore, we define the \emph{generalized Barnes multiple zeta function} by:
 \begin{align*}
 \zeta _N(s,\omega |\mathbf{a}_N;\boldsymbol{\sigma }_N):=\sum_{n_1\ge 0,...,n_N\ge 0}{\frac{\left( \sigma _1 \right) ^{n_1}\cdots \left( \sigma _N \right) ^{n_N}}{(\omega +n_1a_1+\cdots +n_Na_N)^s},}\quad (\Re (\omega )>\Re (s)>N),
\end{align*}
where $\boldsymbol{\sigma }_N=\left( \sigma _1,\ldots,\sigma _N \right) \in \left\{ \pm 1 \right\} ^N$. We also derive a classic integral representation:
\begin{align*}
\zeta _N(s,\omega |\mathbf{a}_N;\boldsymbol{\sigma }_N)=\frac{1}{\Gamma \left( s \right)}\int_0^{\infty}{u^{s-1}}e^{-wu}\prod_{j=1}^N{(}1-\sigma _je^{-a_ju})^{-1}\mathrm{d}u,
\end{align*}
and establish some explicit relations between  mixed Berndt-type integrals and (generalized) Barnes multiple zeta functions. Finally, we establish some closed-form  evaluations of the combinations of several generalized Barnes multiple zeta functions.

\section{Some Definitions and Lemmas}
Adopting Ramanujan's notations, we let
\begin{align}\label{notations-Ramanujan}
	x:=k^2,\ y\left( x \right):=\pi K'/K,\ q\equiv q\left( x \right):=e^{-y},\ z:=z\left( x \right) =2K/\pi,\ z'=dz/dx.
\end{align}
where the \emph{complete elliptic integral of the first kind} is defined by
$$
K\equiv K(k): =\int_0^{\pi /2}{\frac{\mathrm{d}\varphi}{\sqrt{1-k^2\sin ^2\varphi}}=\frac{\pi}{2}}\,\,_2F_1\left(
\frac{1}{2},\frac{1}{2};1;k^2 \right) ,
$$
additionally, let $k'=\sqrt{1-k^2}$, and define the associated \emph{complete elliptic integral} by $$
K'\equiv K(k'): =\int_0^{\pi /2}{\frac{\mathrm{d}\varphi}{\sqrt{1-k'^2\sin ^2\varphi}}=\frac{\pi}{2}}\,\,_2F_1\left( \frac{1}{2},\frac{1}{2};1;k'^2 \right) .
$$
\begin{defn} 
	Let $m\in\mathbb{N}$ and $p\in\mathbb{Z}$. Define
	\begin{align*}
		S_{p,m}(y)&:=\sum_{n = 1}^{\infty}\frac{n^p}{\sinh^m(ny)},&\overline{S}_{p,m}(y)&:=\sum_{n = 1}^{\infty}\frac{n^p}{\sinh^m(ny)}(-1)^{n - 1},\\
		S_{p,m}'(y)&:=\sum_{n = 1}^{\infty}\frac{(2n - 1)^p}{\sinh^m((2n - 1)y/2)},&\overline{C}_{p,m}(y)&:=\sum_{n = 1}^{\infty}\frac{n^p}{\cosh^m(ny)}(-1)^{n - 1},\\
		\mathrm{DS}_{p,m}'(y)&:=\sum_{n=1}^{\infty}{\frac{(2n-1)^p\cosh((2n-1)y/2)}{\sinh ^m((2n-1)y/2))}},&\mathrm{DS}_{p,m}(y) &:=\sum_{n=1}^{\infty}{\frac{n^p\cosh \left( ny \right)}{\sinh ^m(ny)}}.
	\end{align*}
\end{defn}

\begin{lem}\emph{(\cite[Lem. 7.3.]{XZ2023})}\label{Taylor sn}
	For any integer $m\ge 0$ there is a polynomial $g_m(X,Y,Z)\in \mathbb{Z}[X,Y,Z]$ such that 	\begin{align}
&	\frac{\mathrm{d}^m}{\mathrm{d}u^m}\mathrm{sn}\left( u \right) =g_m\left( \sin \varphi ,\cos \varphi ,k^2 \right) \cdot \begin{cases}
		\left( 1-k^2\sin ^2\varphi \right) ^{1/2}, if\,\,m\,\,is\,\,odd;\\
		1,     if\,\,m\,\,is\,\,even,\\
	\end{cases}
\end{align}
	
				\end{lem}
	\begin{lem}\emph{(\cite[Thm. 2.4]{XuZhao-2024})}\label{lem-2,transform}
	Let $x, y, z$ and $z'$ satisfy \eqref{notations-Ramanujan}. Given the formula $\Omega(x,e^{-y},z,z')=0$, we have the transformation formulas
		\begin{align}\label{lem-for-one}
	&\Omega \left( 1-x,e^{-\pi ^2/y},yz/\pi ,\frac{1}{\pi}\left( \frac{1}{x\left( 1-x \right) z}-yz' \right) \right) =0,
	\\
	&\Omega \left( \frac{x}{x-1},-e^{-y},z\sqrt{1-x},\left( 1-x \right) ^{\tfrac{3}{2}}\left( \frac{z}{2}-\left( 1-x \right) z' \right) \right) =0,
	\\
	&\Omega \left( \left( \frac{1-\sqrt{1-x}}{1+\sqrt{1-x}} \right) ^2,e^{-2y},\frac{z\left( 1+\sqrt{1-x} \right)}{2},\frac{\left( 1+\sqrt{1-x} \right) ^3}{4\left( 1-\sqrt{1-x} \right)}\left( \left( 1-x+\sqrt{1-x} \right)z'-\frac{z}{2} \right) \right) =0.\label{lem-for-trans-3}
		\end{align}
	\end{lem}
		\begin{lem} \emph{(\cite{XZ2023})} \label{ExpandS-C}
			Let $n\in \mathbb{Z}$. We have
	\begin{align}
			\frac{\left( -1 \right) ^n}{\sin \left( \frac{1\pm i}{2}z \right)}\overset{z\rightarrow \left( 1\mp i \right) n\pi}{=}\frac{1\mp i}{z-\left( 1\mp i \right) n\pi}+2\sum_{k=1}^{\infty}{\frac{\overline{\zeta }\left( 2k \right)}{\pi ^{2k}}\left( \frac{1\pm i}{2} \right) ^{2k-1}\left( z-\left( 1\mp i \right) n\pi \right) ^{2k-1}},
			\\
			\frac{\left( -1 \right) ^{n-1}}{\cos \left( \frac{1\pm i}{2}z \right)}\overset{z\rightarrow \left( 1 \mp i \right) \tilde{n}\pi}{=}\frac{1\mp i}{z-\left( 1\mp i \right) \tilde{n}\pi}-2\sum_{k=1}^{\infty}{\frac{\overline{\zeta }\left( 2k \right)}{\pi ^{2k}}\left( \frac{1\pm i}{2} \right) ^{2k-1}\left( z-\left( 1\mp i \right) \tilde{n}\pi \right) ^{2k-1}},
			\end{align}
			where $\tilde{n}:=n-\frac{1}{2}$. Here  $\overline{\zeta }(s)$ stands for the alternating Riemann zeta function, which is defined by
			$$
	 \bar{\zeta}\left( s \right) :=\sum_{n=1}^{\infty}{\frac{\left( -1 \right) ^{n-1}}{n^s}}\quad \left( \Re \left( s \right) >0 \right).
			$$
				\end{lem}
	\begin{lem}\emph{(\cite[Lem. 1]{XuZhao-2024})}\label{lem-2,C,S}
		Let $p$ be a positive integer and $
		\alpha ,\beta
		$ be real numbers such as $\alpha \beta=\pi^2$
\begin{align}
&\alpha ^{2p+1}\bar{C}_{2p,2}\left( \alpha \right) -\left( -1 \right) ^pp\frac{\pi ^{2p}}{2^{2p-2}}S'_{2p-1,1}\left( \beta \right)
+\left( -1 \right) ^p\frac{\pi ^{2p}}{2^{2p}}\beta
\mathrm{DS}'_{2p,2}\left( \beta \right)
=0,\\
&\alpha ^{2p}\bar{S}_{2p,2}\left( \alpha \right) -2p\left( -1 \right) ^{p-1}\pi ^{2p-2}\beta S_{2p-1,1}\left( \beta \right)
-\left( -1 \right) ^p\pi ^{2p-2}\beta ^2\mathrm{DS}_{2p,2}\left( \beta \right) -\delta _p=0,
	\end{align}
where $\delta_1=\frac{1}{2}$ and $\delta_p=0$ if $p\ge 2$.
	\end{lem}
	\begin{lem}\emph{(\cite[Thm. 2]{A1976})}\label{sn cofficient symmetry}
	The Maclaurin series of ${\rm sn}(u)$ has the form
	\begin{align}\label{sd-exped-one}
		\mathrm{sn}\left( u \right) =\sum_{n=0}^{\infty}{\mathrm{q}_{2n+1}\left( x \right) \frac{\left( -1 \right) ^nu^{2n+1}}{\left( 2n+1 \right) !}},
	\end{align}	
	where $\text{q}_{2n+1}(x)=\sum^{n}_{j=0}a_jx^j$ and $a_j=a_{n-j}\ (j=0,\ldots,n$).
		\end{lem}

	\begin{lem}\label{either an integer or an integer}
 For any positive even integer $k$, we have
 \begin{align}\label{even integer}
 \left( \sqrt{2}+1 \right) ^k+\left( \sqrt{2}-1 \right) ^k\in \mathbb{Z} ,
 \quad 
 \sqrt{2}\left[ \left( \sqrt{2}+1 \right) ^k-\left( \sqrt{2}-1 \right) ^k \right] \in \mathbb{Z} ,
\end{align}
 and  for any positive odd integer $k$, we have 
  \begin{align}\label{odd integer}
 	\left( \sqrt{2}+1 \right) ^k-\left( \sqrt{2}-1 \right) ^k\in \mathbb{Z} ,
  \quad 
 	\sqrt{2}\left[ \left( \sqrt{2}+1 \right) ^k+\left( \sqrt{2}-1 \right) ^k \right] \in \mathbb{Z} .
 \end{align}
\end{lem}
\begin{proof}
	Expanding the expressions in \eqref{even integer} and \eqref{odd integer} via the binomial theorem proves the lemma.
\end{proof}

\section{Mixed Berndt-Type Integrals via Hyperbolic Series}
In this section, we  establish some explicit relations between  mixed Berndt-type integrals and some hyperbolic summations  using the contour integrations.
\begin{thm}\label{thm1cosh++}
For any  $a\ge 1$, we have
\begin{align}\label{onecosh++}
\int_0^{\infty}{\frac{x^a\left( \sinh x-i\sin x \right) \mathrm{d}x}{\left[ \sinh ^2x+\sin ^2x \right] \left[ \cosh x+\cos x \right]}}&-i^a\int_0^{\infty}{\frac{x^a\left( \sin x-i\sinh x \right) \mathrm{d}x}{\left[ \sinh ^2x+\sin ^2x \right] \left[ \cosh x+\cos x \right]}}
\nonumber\\
&=\left( \frac{1+i}{2}\pi \right) ^{a+1}\sum_{n=1}^{\infty}{\frac{\left( -1 \right) ^nn^a}{\cosh ^2\left[ \tfrac{n\pi}{2} \right]}}.
\end{align}	
\end{thm}
\begin{proof}
Let $z=x+iy$ with $x,y\in \R$. Consider
\begin{align}\label{complex equation,cosh,1/2}
\mathscr{I} _a=\lim_{R\rightarrow \infty} \int_{\mathrm{C}_R}^{}{\frac{z^a\mathrm{d}z}{\left( \sinh z+i\sin z \right) \left( \cosh z+\cos z \right)}}=\lim_{R\rightarrow \infty} \int_{\mathrm{C}_R}^{}{F_1\left( z \right)}\mathrm{d}z,
\end{align}
where $C_R$ denotes the quarter- circular contour consisting of the interval $[0,R]$, the quarter-circle $\Gamma_R$ with $|z|=R,0 \le \arg (z)\le \pi/2$, and $[iR,0]$,  see the following figure:
\begin{center}
\includegraphics[height=2in]{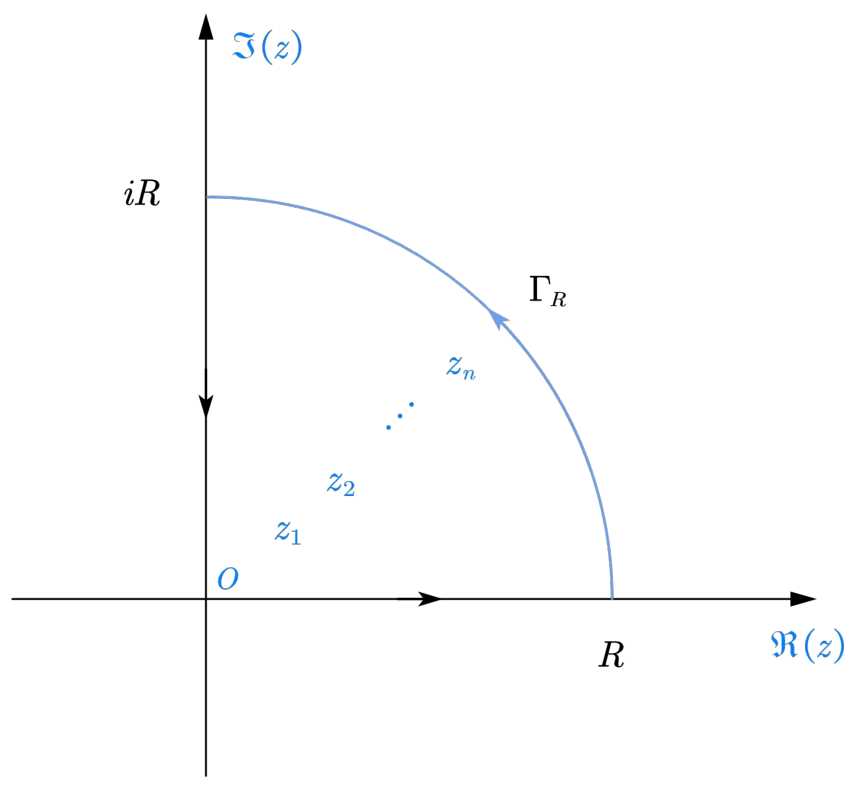}
\end{center}
Obviously, there exist poles when
$$
\left( \sinh z+i\sin z \right) \left( \cosh z+\cos z \right) =2i\cos ^2\left\{ \frac{1+i}{2}z \right\} \sin \left\{ \left( 1-i \right) z \right\} =0.
$$
The only poles lying inside $C_R$ are $z_n=\frac{n(1+i)\pi}{2}$ for $ n\ge 1$ with $ |z_n|< R$. Using Lemma \ref{ExpandS-C}, the residue $\text{Res}[F_1(z),z_n]$ at each such pole is given by
\begin{align}
 & \mathrm{Res}\left[ F\left( z \right) ,z_n \right] =\frac{\left( -1 \right) ^n\left[ \tfrac{1+i}{2}n\pi \right] ^a}{2i\left( 1-i \right) \cos ^2\left[ \left( \tfrac{1+i}{2} \right) ^2n\pi \right]}
\nonumber\\
&=\left( \frac{1+i}{2} \right) ^{a+1}\frac{\pi ^a}{2i}\frac{\left( -1 \right) ^nn^a}{\cos^2 \left[ \tfrac{i}{2}n\pi \right]}=\left( \frac{1+i}{2} \right) ^{a+1}\frac{\pi ^a}{2i}\frac{\left( -1 \right) ^nn^a}{\cosh^2 \left[ \tfrac{n\pi}{2} \right]}.
\end{align}
As $R\rightarrow \infty$, it is easy to show that
$$
\int_{\Gamma _R}^{}{\frac{z^a\mathrm{d}z}{\left( \sinh z+i\sin z \right) \left( \cosh z+\cos z \right)}}=o\left( 1 \right) .
$$
Applying the residue theorem and letting $R\rightarrow \infty$, we conclude that
\begin{align*}
&2\pi i\sum_{n=1}^{\infty}{\mathrm{Res}\left[ F_1\left( z \right) ,z_n \right]}=\lim_{R\rightarrow \infty} \int_{\mathrm{C}_R}^{}{\frac{z^a\mathrm{d}z}{\left( \sinh z+i\sin z \right) \left( \cosh z+\cos z \right)}}
\\
&=\int_0^{\infty}{\frac{x^a\mathrm{d}x}{\left[ \sinh x+i\sin x \right] \left[ \cosh x+\cos x \right]}}+i\int_0^{\infty}{\frac{\left( ix \right) ^a\mathrm{d}x}{\left[ \sinh \left( ix \right) +i\sin \left( ix \right) \right] \left[ \cosh \left( ix \right) +\cos \left( ix \right) \right]}}
\\
&=\int_0^{\infty}{\frac{x^a\left( \sinh x-i\sin x \right) \mathrm{d}x}{\left[ \sinh ^2x+\sin ^2x \right] \left[ \cosh x-\cos x \right]}}-i^a\int_0^{\infty}{\frac{x^a\mathrm{d}x}{\left[ \sin x+i\sinh x \right] \left[ \cosh x+\cos x \right]}}.
\end{align*}
Thus, by combining the related identities, we deduce the desired result.
\end{proof}

\begin{re}
Fix the principal branch of $z^a$ on the sector $0<\arg z<\pi/2$.
For $0<\Re(a)<1$, indent the contour at the origin by the small arc
$\Gamma_\varepsilon=\{\varepsilon e^{it}:0\le t\le\pi/2\}$, see the following figure.
Near $z=0$ one has $F_1(z)=O(|z|^{\Re(a)-1})$, hence
\[
\Bigl|\int_{\Gamma_\varepsilon}F_1(z)\,dz\Bigr|
=O(\varepsilon^{\Re(a)})\to 0, \quad  \text{as} \,\varepsilon\to0.
\]
The large circular arc still gives $o(1)$. Therefore the
left-hand side of \eqref{onecosh++} defines an analytic function of $a$ on $\Re(a)>0$.
The right-hand side is entire in $a$ (uniform convergence). Since the identity holds for  $a\ge1$,
the identity theorem extends \eqref{onecosh++} to the half-plane $\Re(a)>0$.	Likewise, by the same argument, \eqref{onecosh-+} extends to $\Re(a)>0$.
\end{re}

\begin{figure}[htbp]
	\centering
	\includegraphics[height=2.2in]{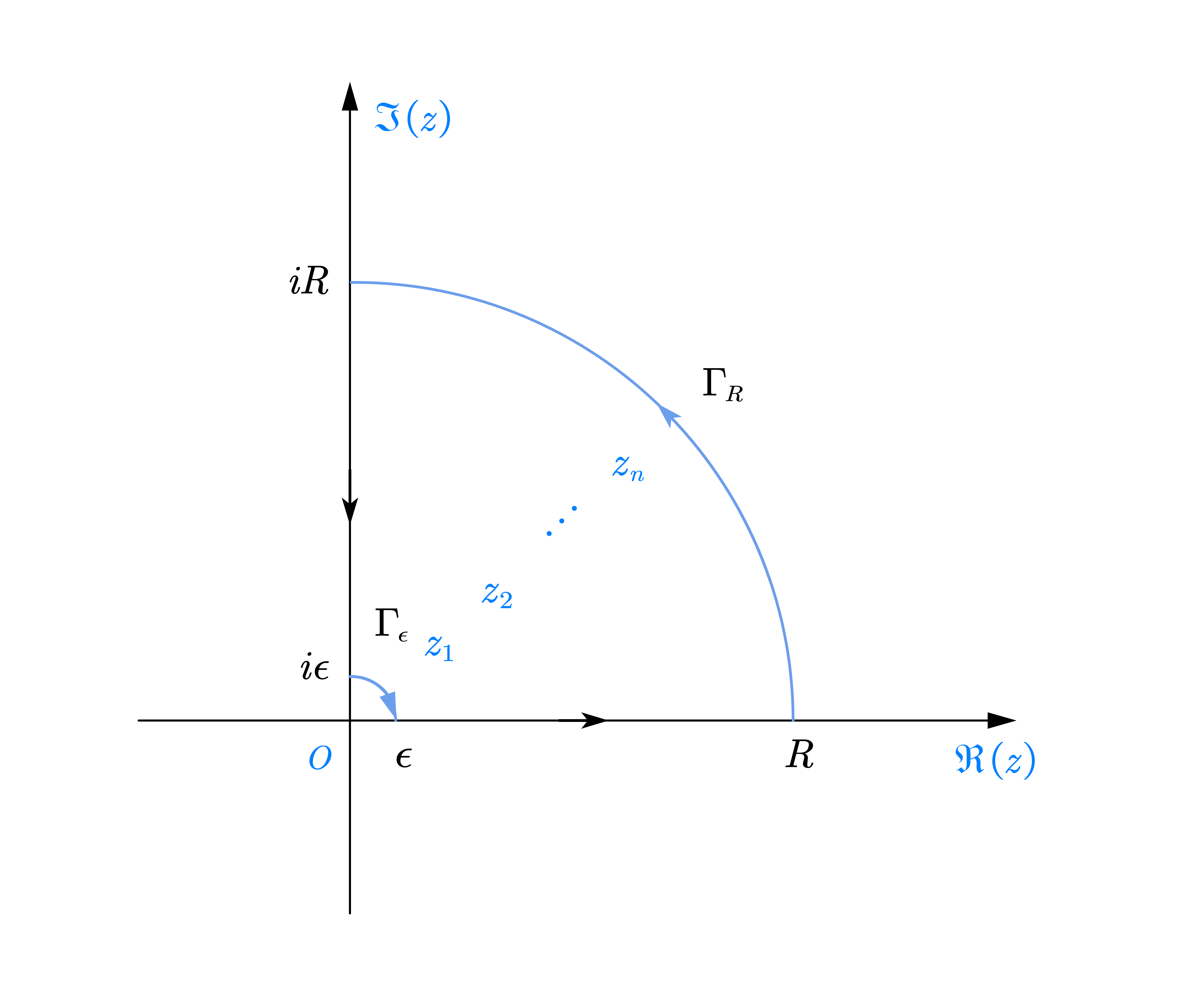} 
	\label{fig:Q1}
\end{figure}

\begin{thm}\label{Sin+,Cos-}
For any  $a\ge 3$, we have
\begin{align}\label{twoCosh+-}
\int_0^{\infty}{\frac{x^a\left( \sinh x-i\sin x \right) \mathrm{d}x}{\left[ \sinh ^2x+\sin ^2x \right] \left[ \cosh x-\cos x \right]}}+&\left( -i \right) ^a\int_0^{\infty}{\frac{x^a\left( \sin x-i\sinh x \right) \mathrm{d}x}{\left[ \sinh ^2x+\sin ^2x \right] \left[ \cosh x-\cos x \right]}}
\nonumber\\
&
=\left( \frac{1-i}{2}\pi \right) ^{a+1}\sum_{n=1}^{\infty}{\frac{\left( -1 \right) ^nn^a}{\sinh ^2\left( \frac{n\pi}{2} \right)}}.
\end{align}	
\end{thm}
\begin{proof}
We consider
\begin{align}\label{equation,sinh,1/2}
\mathscr{J} _a=\lim_{R\rightarrow \infty} \int_{\mathrm{C}_R}^{}{\frac{z^a\mathrm{d}z}{\left( \sinh z+i\sin z \right) \left( \cosh z-\cos z \right)}}=\lim_{R\rightarrow \infty} \int_{\mathrm{C}_R}^{}{F_2\left( z \right)}\mathrm{d}z,
\end{align}
where $C_R$ denotes the quarter-circular contour consisting of the interval $[0,R]$, the quarter-circle $\Gamma_R$ with $|z|=R\ (-\pi/2 \le \arg (z)\le 0)$, and $[-iR,0]$, see the following figure.
\begin{center}
\includegraphics[height=2in]{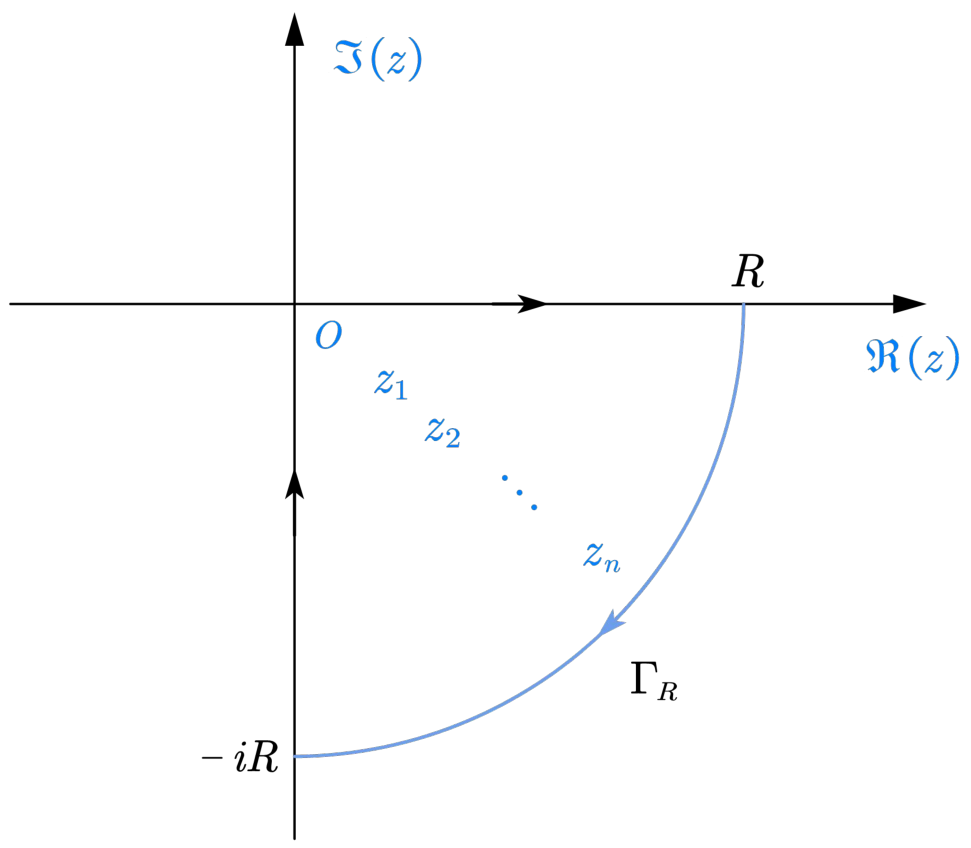}
\end{center}
Similar to Theorem 3.1, we derive that the only poles inside $C_R$ are $z_n=\frac{n(1-i)\pi}{2}$ for $ n\ge 1$ with $ |z_n|< R$, and the residue $\text{Res}[F_2(z),z_n]$ at each pole is given by
\begin{align}
&\mathrm{Res}\left[ F_2\left( z \right) ,z=\frac{1-i}{2}n\pi \right]=\left( \frac{1-i}{2} \right) ^{a+1}\frac{\pi ^a}{2i}\frac{\left( -1 \right)^{n-1}n^a}{\sinh ^2\left( \frac{n\pi}{2} \right)}.
\end{align}
It is easy to show that, letting $R\rightarrow \infty$,
$$
\int_{\Gamma _R}^{}{\frac{z^a\mathrm{d}z}{\left( \sinh z+i\sin z \right) \left( \cosh z-\cos z \right)}}=o\left( 1 \right) .
$$
Applying the residue theorem and letting $R$ tend to infinity, we conclude that
\begin{align*}
&-2\pi i\sum_{n=1}^{\infty}{\mathrm{Res}\left[ F_2\left( z \right) ,z_n \right]}=\lim_{R\rightarrow \infty} \int_{\mathrm{C}_R}^{}{\frac{z^a\mathrm{d}z}{\left( \sinh z+i\sin z \right) \left( \cosh z-\cos z \right)}}
\\
&=\int_0^{\infty}{\frac{x^a\left( \sinh x-i\sin x \right) \mathrm{d}x}{\left[ \sinh ^2x+\sin ^2x \right] \left[ \cosh x-\cos x \right]}}+
\int_0^{\infty}{\frac{(-i)^ax^a\left( \sin x-i\sinh x \right) \mathrm{d}x}{\left[ \sinh ^2x+\sin ^2x \right] \left[ \cosh x-\cos x \right]}}.
\end{align*}
Thus, the proof is complete.
\end{proof}

\begin{re}
Fix the principal branch of $z^a$ on the sector $-\pi/2<\arg z<0$. For $2<\Re(a)<3$, indent the contour at the origin by the small arc
$\Gamma_\varepsilon=\{\varepsilon e^{it}:-\pi/2\le t\le0\}$, see the following figure.
Near $z=0$ the integrand in \eqref{twoCosh+-} satisfies $F_2(z)=O(|z|^{\Re(a)-3})$, hence
\[
\Bigl|\int_{\Gamma_\varepsilon}F_2(z)\,dz\Bigr|
=O(\varepsilon^{\Re(a)-2})\to 0, \quad \text{as} \,\varepsilon\to0,
\]
and the large arc still contributes $o(1)$. Therefore the
left-hand side of \eqref{twoCosh+-} defines an analytic function of $a$ on $\Re(a)>2$.
The right-hand side is entire in $a$ (uniform convergence).
Since the identity holds for all real $a\ge3$, the identity theorem extends \eqref{twoCosh+-} to the
half-plane $\Re(a)>2$. \emph{The case of \eqref{sincosh--} follows by the same argument.}
\end{re}
\begin{figure}[htbp]
	\centering
	\includegraphics[height=2.2in]{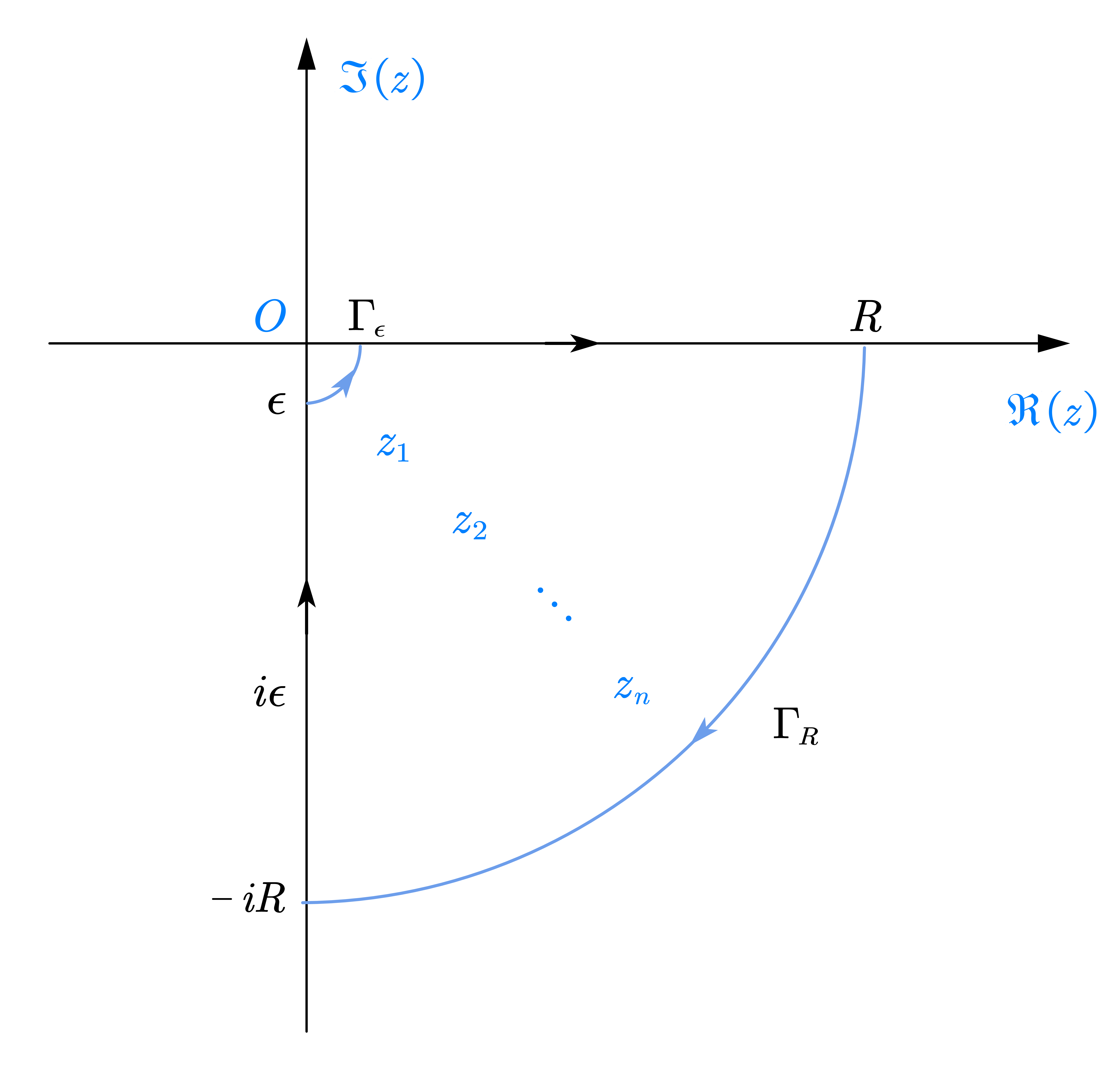} 
	\label{fig:Q2}
\end{figure}

\begin{thm}\label{Sin-,Cos-}
For any  $a\ge3$, we have
\begin{align}\label{sincosh--}
\int_0^{\infty}{\frac{x^a\mathrm{d}x}{\left[ \sinh x-i\sin x \right] \left[ \cosh x-\cos x \right]}}-\left( -i \right) ^{a+1}\int_0^{\infty}{\frac{x^a\mathrm{d}x}{\left[ \sinh x+i\sin x \right] \left[ \cosh x-\cos x \right]}}
\nonumber\\
=2\left( 1-i \right) ^{a-1}\pi ^a\sum_{n=1}^{\infty}{\left[ \frac{an^{a-1}}{\sinh \left( 2n\pi \right)}-2\pi \frac{n^a\cosh \left( 2n\pi \right)}{\sinh ^2\left( 2n\pi \right)} \right]}.
\end{align}	
\end{thm}
\begin{proof}
Consider
\begin{align}\label{equations ,sinh,even}
\mathscr{H} _a=\lim_{R\rightarrow \infty} \int_{\mathrm{C}_R}^{}{\frac{z^a\mathrm{d}z}{\left( \sinh z-i\sin z \right) \left( \cosh z-\cos z \right)}},
\end{align}
where $C_R$ denotes the same positively oriented quarter-circular contour that we used in the proof of Theorem \ref{Sin+,Cos-}. 	
The proof is completely similar as the proof of Theorem \ref{Sin+,Cos-}, so we leave it to the interested reader.
\end{proof}

\begin{thm}\label{Sin-,Cos+}
For any  $a\ge 1$, we have
\begin{align}\label{onecosh-+}
\int_0^{\infty}&{\frac{x^a\mathrm{d}x}{\left[ \sinh x-i\sin x \right] \left[ \cosh x+\cos x \right]}}-\int_0^{\infty}{\frac{i^{a+1}x^a\mathrm{d}x}{\left[ \sinh x+i\sin x \right] \left[ \cosh x+\cos x \right]}}
\nonumber\\
&
\quad\quad	=\frac{\left( 1+i \right) ^{1+a}i}{2^{a-1}}\pi ^a\sum_{n=0}^{\infty}{\left[ \frac{a\left( 2n+1 \right) ^{a-1}}{\sinh \left( \left( 2n+1 \right) \pi \right)}-\pi \frac{\left( 2n+1 \right) ^a\cosh \left( \left( 2n+1 \right) \pi \right)}{\sinh ^2\left( \left( 2n+1 \right) \pi \right)} \right]}.
\end{align}
\end{thm}
\begin{proof}
Consider
\begin{align}\label{equations ,sinh,odd}
\mathscr{W} _a=\lim_{R\rightarrow \infty} \int_{\mathrm{C}_R}^{}{\frac{z^a\mathrm{d}z}{\left( \sinh z-i\sin z \right) \left( \cosh z+\cos z \right)}},
\end{align}
where $C_R$ denotes the same positively oriented quarter-circular contour that we used in the proof of Theorem \ref{thm1cosh++}.
The proof is completely similar as the proof of Theorem \ref{thm1cosh++}, we leave the details to the interested reader.
\end{proof}

Setting $a=2\mu+1$ and $2\mu$ (or alternatively $2\mu+2$) in Theorems \ref{thm1cosh++}-\ref{Sin-,Cos+} respectively, yields the following corollaries.
\begin{cor}
	For $\mu\in \N$, we have
	\begin{align}
		\int_0^{\infty}{\frac{\left[ 1+\left( -1 \right) ^{\mu +1} \right] x^{2\mu +1}\sinh x\mathrm{d}x}{\left[ \sinh ^2x+\sin ^2x \right] \left[ \cosh x+\cos x \right]}}+&i\int_0^{\infty}{\frac{\left[ \left( -1 \right) ^{\mu +1}-1 \right] x^{2\mu +1}\sin x\mathrm{d}x}{\left[ \sinh ^2x+\sin ^2x \right] \left[ \cosh x+\cos x \right]}}
		\nonumber\\
		&=-\left( \frac{i}{2}\pi ^2 \right) ^{\mu +1} \bar{C}_{2\mu+1,2}\left( \frac{\pi}{2} \right) ,
	\end{align}
	\begin{align}
		\int_0^{\infty}{\frac{\left( 1+\left( -1 \right) ^{\mu +1} \right) x^{2\mu +1}\sinh x\mathrm{d}x}{\left[ \sinh ^2x+\sin ^2x \right] \left[ \cosh x-\cos x \right]}}+&i\int_0^{\infty}{\frac{\left( \left( -1 \right) ^{\mu +1}-1 \right) x^{2\mu +1}\sin x\mathrm{d}x}{\left[ \sinh ^2x+\sin ^2x \right] \left[ \cosh x-\cos x \right]}}
		\nonumber\\
		&=-\left( -\frac{i}{2}\pi ^2 \right) ^{\mu +1}\bar{S}_{2\mu+1,2}\left( \frac{\pi}{2} \right),
	\end{align}
	\begin{align}\label{sin+,cos-,odd}
		\int_0^{\infty}{\frac{\left( 1+\left( -1 \right) ^{\mu} \right) x^{2\mu +1}\sinh x\mathrm{d}x}{\left[ \sinh ^2x+\sin ^2x \right] \left[ \cosh x+\cos x \right]}+}i\int_0^{\infty}{\frac{\left( 1-\left( -1 \right) ^{\mu} \right) x^{2\mu +1}\sin x\mathrm{d}x}{\left[ \sinh ^2x+\sin ^2x \right] \left[ \cosh x+\cos x \right]}}
		\nonumber\\
		=
		\frac{i^{\mu +2}\pi ^{2\mu +1}}{2^{\mu -1}}\left[ \left( 2\mu +1 \right) S'_{2\mu ,1}\left( 2\pi \right) -\pi \mathrm{DS}'_{2\mu +1,2}\left( 2\pi \right) \right] ,
	\end{align}
	\begin{align}
		\int_0^{\infty}{\frac{\left( 1+\left( -1 \right) ^{\mu} \right) x^{2\mu +1}\sinh x\mathrm{d}x}{\left[ \sinh ^2x+\sin ^2x \right] \left[ \cosh x-\cos x \right]}}+i\int_0^{\infty}{\frac{\left( 1+\left( -1 \right) ^{\mu +1} \right) x^{2\mu +1}\sin x\mathrm{d}x}{\left[ \sinh ^2x+\sin ^2x \right] \left[ \cosh x-\cos x \right]}}
		\nonumber\\
		=
		\left( -i \right) ^{\mu}2^{\mu +1}\pi ^{2\mu +1}{\left[ \left( 2\mu +1 \right) S_{2\mu ,1}\left( 2\pi \right) -2\pi \mathrm{DS}_{2\mu +1,2}\left( 2\pi \right) \right]}.
	\end{align}
	
\end{cor}

\begin{cor}
	For $\mu\in \N$, we have
	\begin{align}
		\int_0^{\infty}{\frac{x^{2\mu}\left( \sinh x-i\sin x \right) \mathrm{d}x}{\left[ \sinh ^2x+\sin ^2x \right] \left[ \cosh x+\cos x \right]}}-&\left( -1 \right) ^{\mu}\int_0^{\infty}{\frac{x^{2\mu}\left( \sin x-i\sinh x \right) \mathrm{d}x}{\left[ \sinh ^2x+\sin ^2x \right] \left[ \cosh x+\cos x \right]}}
		\nonumber\\
		&=-\left( \frac{1+i}{2}\pi \right) ^{2\mu +1}\bar{C}_{2\mu,2}\left( \frac{\pi}{2} \right),
	\end{align}
	\begin{align}
		\int_0^{\infty}{\frac{x^{2\mu+2}\left( \sinh x-i\sin x \right) \mathrm{d}x}{\left[ \sinh ^2x+\sin ^2x \right] \left[ \cosh x-\cos x \right]}}+&\left( -1 \right) ^{\mu+1}\int_0^{\infty}{\frac{x^{2\mu+2}\left( \sin x-i\sinh x \right) \mathrm{d}x}{\left[ \sinh ^2x+\sin ^2x \right] \left[ \cosh x-\cos x \right]}}
		\nonumber\\
		&=-\left( \frac{1-i}{2}\pi \right) ^{2\mu +3}\bar{S}_{2\mu+2,2}\left( \frac{\pi}{2} \right),
	\end{align}
	\begin{align}
		\int_0^{\infty}{\frac{x^{2\mu}\left( \sinh x+i\sin x \right) \mathrm{d}x}{\left[ \sinh ^2x+\sin ^2x \right] \left[ \cosh x+\cos x \right]}}-\left( -1 \right) ^{\mu}i\int_0^{\infty}{\frac{x^{2\mu}\left( \sinh x-i\sin x \right) \mathrm{d}x}{\left[ \sinh ^2x+\sin ^2x \right] \left[ \cosh x+\cos x \right]}}
		\nonumber\\
		\quad \quad =
		\frac{\left( 1+i \right) ^{2\mu +1}i}{2^{2\mu -1}}\pi ^{2\mu}\left[ 2\mu S'_{2\mu -1,1}\left( 2\pi \right) -\pi \mathrm{DS}'_{2\mu ,2}\left( 2\pi \right) \right],
	\end{align}	
	\begin{align}
		\int_0^{\infty}{\frac{x^{2\mu+2}\left( \sinh x+i\sin x \right) \mathrm{d}x}{\left[ \sinh ^2x+\sin ^2x \right] \left[ \cosh x-\cos x \right]}}+\left( -1 \right) ^{\mu+1}i\int_0^{\infty}{\frac{x^{2\mu+2}\left( \sinh x-i\sin x \right) \mathrm{d}x}{\left[ \sinh ^2x+\sin ^2x \right] \left[ \cosh x-\cos x \right]}}
		\nonumber\\=
		4\left( 1-i \right) ^{2\mu+1}\pi ^{2\mu+2} \left[ (\mu+1) S_{2\mu+1,1}\left( 2\pi \right) -\pi \mathrm{DS}_{2\mu+2 ,2}\left( 2\pi \right) \right].
	\end{align}
\end{cor}

\section{Evaluations of Hyperbolic Summations via Jacobi Functions}

Now, we evaluate the  reciprocal hyperbolic series of Ramanujan type
$$
\bar{C}_{2p,2}\left( y \right)=\sum_{n=1}^{\infty}{\frac{n^{2p}\left( -1 \right) ^{n-1}}{\cosh ^2\left( ny \right)}}\quad\text{and}\quad\bar{S}_{2p,2}\left( y \right) =\sum_{n=1}^{\infty}{\frac{n^{2p}\left( -1 \right) ^{n-1}}{\sinh ^2\left( ny \right)}}
$$
by applying the Fourier series expansions and the Maclaurin series expansions of certain Jacobi elliptic functions. Let $x,y$ and $z$ satisfy the relations in \eqref{notations-Ramanujan}.

\begin{thm}\label{barC2p,2}
Let $p\in \mathbb{N}$, we have
\begin{align}
\bar{C}_{2p,2}\left( y \right)=&\left( -1 \right) ^p\frac{pz^{2p+1}z'x(1-x)}{2^{2p-1}}\mathrm{q}_{2p-1}\left( 1-x \right) \sqrt{1-x}\nonumber\\
& +\left( -1 \right) ^p\frac{z^{2p+2}x(1-x)}{2^{2p}}\frac{\mathrm{d}}{\mathrm{d}x}\left[ \mathrm{q}_{2p-1}\left( 1-x \right) \sqrt{1-x} \right],
\end{align}
where ${\rm q}_{2p-1}\left( x \right)$ appears in the coefficients of the Maclaurin series of $\mathrm{sn}\left( u \right)$.
\end{thm}
\begin{proof}
Setting $\alpha=y$ and $\beta=\tfrac{\pi ^2}{y}$ in Lemma \ref{lem-2,C,S} gives
\begin{align}\label{equ-lemma3-one}
\bar{C}_{2p,2}\left( y \right) =\frac{\left( -1 \right) ^p}{y^{2p+1}}\frac{p\pi ^{2p}}{2^{2p-2}}S_{2p-1,1}^{\prime}\left( \frac{\pi ^2}{y} \right) -\left( -1 \right) ^p\frac{\pi ^{2p+2}}{2^{2p}y^{2p+2}}\text{DS}'_{2p,2}\left( \frac{\pi ^2}{y} \right).
\end{align}
Now, we need to compute the series $S_{2p-1,1}^{'}\left( y \right)$. According to Lemma \ref{sn cofficient symmetry}, there exist polynomials  ${\rm q}_{2n-1}(x)\in\Z[x]$ such that
the Maclaurin series of ${\rm sn}(u)$ has the form
\begin{align}\label{sd-exped-one}
\mathrm{sn}\left( u \right) =\sum_{n=0}^{\infty}{\mathrm{q}_{2n+1}\left( x \right) \frac{\left( -1 \right) ^nu^{2n+1}}{\left( 2n+1 \right) !}}.
\end{align}	
From \cite{DCLMRT1992}, we recall the Fourier series expansions of Jacobi elliptic functions ${\rm sn}(u)$
\begin{align}\label{sd}
&{\rm sn}(u)=\frac{2\pi}{Kk}\sum_{n=0}^{\infty}\frac{q^{n+1/2}}{1-q^{2n+1}}\sin\left[(2n+1)\frac{\pi u}{2K}\right]\quad (|q|<1).
\end{align}
Setting  $q\equiv q\left( x \right)=e^{-y}$, we then have
\begin{align}\label{sd-one}
{\rm sn}(u)&=\frac{2\pi}{Kk}\sum_{n=0}^{\infty}{\frac{q^{n+\tfrac{1}{2}}\left( 2n+1 \right) ^{2j+1}}{1-q^{2n+1}}}\sum_{j=0}^{\infty}{\frac{\left( -1 \right) ^j\left( \frac{\pi u}{2K} \right) ^{2j+1}}{\left( 2j+1 \right) !}}
	 \nonumber\\
	 &=\frac{\pi}{Kk}\sum_{j=0}^{\infty}{\frac{\left( -1 \right) ^j\left( \frac{\pi u}{2K} \right) ^{2j+1}}{\left( 2j+1 \right) !}}S_{2j+1,1}^{\prime}\left( y \right) .
\end{align}	
Comparing the coefficients of $u^{2m+1}$ in \eqref{sd-exped-one} and \eqref{sd-one}, we can deduce
\begin{align}
S_{2n+1,1}'\left( y \right)=\sum_{j=1}^{\infty}{\frac{\left( 2j-1 \right) ^{2n+1}}{\sinh \left( \frac{2j-1}{2}y \right)}}
=z^{2n+2}\sqrt{x}\frac{\mathrm{q}_{2n+1}\left( x \right)}{2}. \label{S'2n+1,1}
\end{align}
Differentiating \eqref{S'2n+1,1} with respect to $y$ yields
\begin{align*}
\frac{\mathrm{d}}{\mathrm{d}y}S_{2p-1,1}^{\prime}\left( y \right) =\sum_{n=1}^{\infty}{\frac{\mathrm{d}}{\mathrm{d}y}\frac{\left( 2n-1 \right) ^{2p-1}}{\sinh \left( \frac{2n-1}{2}y \right)}=-\frac{1}{2}}\text{DS}'_{2p,2}\left( y \right).
\end{align*}
Using the relation $\frac{dx}{dy}=-x(1-x)z^2$ (see \cite[P. 120, Entry. 9(i)]{B1991}), we obtain
\begin{align}\label{equ-relation-differ-one}
\text{DS}'_{2p,2}\left( y \right)=2x\left( 1-x \right) z^2\frac{\mathrm{d}}{\mathrm{d}x}S_{2p-1,1}^{\prime}\left( y \right).
\end{align}
Applying Lemma \ref{lem-2,transform}, for $\Omega \left( x,e^{-y},z,z' \right) =0$, we have
$$\Omega \left( 1-x,e^{-\pi ^2/y},yz/\pi ,\frac{1}{\pi}\left( \frac{1}{x\left( 1-x \right) z}-yz' \right) \right) =0.$$
Hence,  \eqref{S'2n+1,1} and \eqref{equ-relation-differ-one} can be rewritten as
\begin{align}\label{equ-lemma3-two}
S_{2p-1,1}^{\prime}\left( \frac{\pi ^2}{y} \right) =\left( \frac{yz}{\pi} \right) ^{2p}\frac{\mathrm{q}_{2p-1}\left( 1-x \right)}{2}\sqrt{1-x}
\end{align}
and
\begin{align}\label{equ-lemma3-three}
\text{DS}'_{2p,2}\left( \frac{\pi ^2}{y} \right)&=2x\left( 1-x \right) \left( \frac{yz}{\pi} \right) ^2\frac{\mathrm{d}}{\mathrm{d}\left( 1-x \right)}S_{2p-1,1}^{\prime}\left( \frac{\pi ^2}{y} \right)\nonumber\\
&=-2x\left( 1-x \right) \left( \frac{yz}{\pi} \right) ^2\frac{\mathrm{d}}{\mathrm{d}x}S_{2p-1,1}^{\prime}\left( \frac{\pi ^2}{y} \right).
\end{align}
Finally, combining \eqref{equ-lemma3-one}, \eqref{S'2n+1,1}, \eqref{equ-lemma3-two} and \eqref{equ-lemma3-three} with elementary calculations yields the desired result.
\end{proof}

Using \emph{Mathematica}, we can easily compute the following using the formulas in Theorem \ref{barC2p,2}.
\begin{cor}\emph{(\cite[Cor. 1]{XuZhao-2024})} We have
\begin{align*}
\bar{C}_{2,2}(y)&=\frac{1}{8}\sqrt{1 - x}xz^{3}[4(x - 1)z'+z],\\
\bar{C}_{4,2}(y)&=\frac{1}{32}\sqrt{1 - x}xz^{5}[8(x^{2}-3x + 2)z'+(3x - 4)z],\\
\bar{C}_{6,2}(y)&=\frac{1}{128}\sqrt{1 - x}xz^{7}[12(x^{3}-17x^{2}+32x - 16)z'+(5x^{2}-52x + 48)z],\\
\bar{C}_{8,2}(y)&=\frac{1}{512}\sqrt{1 - x}xz^{9}\left\{16(x^{4}-139x^{3}+546x^{2}-680x + 272)z'\right.\\
&\quad\quad\quad\quad\left.+ (7x^{3}-696x^{2}+1776x - 1088)z\right\},\\
\bar{C}_{10,2}(y)&=\frac{1}{2048}\sqrt{1 - x}xz^{11}\left\{20(x^{5}-1233x^{4}+10400x^{3}-25040x^{2}+23808x - 7936)z'\right.\\
&\quad\quad\quad\quad\left.+ (9x^{4}-8632x^{3}+53232x^{2}-84288x + 39680)z\right\}.
\end{align*}
\end{cor}

\begin{thm}{\label{barS2p,2}}
Let $p\ge 2$ be a positive integer. We have
\begin{align}\label{barS2p3}
\bar{S}_{2p,2}\left( y \right)& =\left( -1 \right) ^{p-1}x\left( 1-x \right) \frac{z^{2p+2}}{2^{2p}}\frac{\mathrm{d}}{\mathrm{d}x}\left\{ \left( 1-x \right) \mathrm{Q}_{2p-2}\left( 1-x \right) \right\}
\nonumber\\
&\quad+2p\left( -1 \right) ^{p-1}\left( 1-x \right) ^2x\frac{z^{2p+1}z'}{2^{2p}}\mathrm{Q}_{2p-2}\left( 1-x \right),
\end{align}
where $\text{Q}_{2p-2}\left( x \right)$ appears in the coefficients of the Maclaurin series of $\mathrm{sn}^2\left( u \right)$.
\end{thm}
\begin{proof}
According to Lemma \ref{lem-2,C,S}, we have
\begin{align}\label{thm5.3-proof-begin-1}
\bar{S}_{2p,2}\left( y \right) =\frac{2p\left( -1 \right) ^{p-1}\pi ^{2p}}{y^{2p+1}}S_{2p-1,1}\left( \frac{\pi ^2}{y} \right) +\left( -1 \right) ^p\frac{\pi ^{2p+2}}{y^{2p+2}}\text{DS}_{2p,2}\left( \frac{\pi ^2}{y} \right).
\end{align}
From \cite{Greenhill1892}, we have
\begin{align*}
\mathrm{sn}^{2}u&=\left(1 - \frac{E}{K}\right)\frac{1}{k^{2}} - \frac{\pi^{2}}{K^{2}k^{2}}\sum_{n = 1}^{\infty}\frac{n\cos(n\pi u / K)}{\sinh(ny)}\\
&=\left(1 - \frac{E}{K}\right)\frac{1}{k^{2}} - \frac{\pi^{2}}{K^{2}k^{2}}\sum_{j = 0}^{\infty}\frac{(-1)^{j}(\pi u / K)^{2j}}{(2j)!}S_{2j + 1,1}(y),
\end{align*}
where $E$ is the \emph{complete elliptic integrals of the second kinds}. Applying \eqref{sd-exped-one}, we deduce the power series expansion of the function $\mathrm{sn}^2 (u)$
\begin{align*}
\mathrm{sn}^2u&=\left( \sum_{n=0}^{\infty}{\mathrm{q}_{2n+1}\left( x \right) \frac{\left( -1 \right) ^nu^{2n+1}}{\left( 2n+1 \right) !}} \right) ^2
\\
&=\sum_{n=0}^{\infty}{\sum_{j=0}^n{\left( \begin{array}{c}
			2n+2\\
			2j+1\\
		\end{array} \right) \mathrm{q}_{2j+1}\left( x \right) \mathrm{q}_{2n-2j+1}\left( x \right)}}\frac{\left( -1 \right) ^nu^{2n+2}}{\left( 2n+2 \right) !}=\sum_{n=0}^{\infty}{\mathrm{Q}_{2n+2}\left( x \right)}\frac{\left( -1 \right) ^nu^{2n+2}}{\left( 2n+2 \right) !}.
\end{align*}
Comparing the coefficients of $u^{2p-2}$ yields
\begin{align}\label{S_2p-1}
S_{2p-1,1}\left( y \right) =\frac{z^{2p}x}{2^{2p}}\mathrm{Q}_{2p-2}\left( x \right)\quad (p=2,3,4,\ldots).
\end{align}
Differentiating \eqref{S_2p-1}   with respect to $y$ gives
\begin{align}
 \text{DS}_{2p,2}\left( y \right)&=-\frac{\mathrm{d}}{\mathrm{d}y}S_{2p-1,1}\left( y \right) =-\frac{\mathrm{d}}{\mathrm{d}y}\left[ \frac{z^{2p}x}{2^{2p}}\mathrm{Q}_{2p-2}\left( x \right) \right]
\nonumber\\
 &=-\frac{\mathrm{d}x}{\mathrm{d}y}\frac{\mathrm{d}}{\mathrm{d}x}\left[ \frac{z^{2p}x}{2^{2p}}\mathrm{Q}_{2p-2}\left( x \right) \right] =x\left( 1-x \right) z^2\frac{\mathrm{d}}{\mathrm{d}x}\left[ \frac{z^{2p}x}{2^{2p}}\mathrm{Q}_{2p-2}\left( x \right) \right]\nonumber
\\
 &=x\left( 1-x \right) \frac{z^{2p+2}}{2^{2p}}\frac{\mathrm{d}}{\mathrm{d}x}\left[ x\mathrm{Q}_{2p-2}\left( x \right) \right] +2px^2\left( 1-x \right) \frac{z^{2p+1}z'}{2^{2p}}\mathrm{Q}_{2p-2}\left( x \right).\label{thm-prof-series-two}
 	\end{align}
Applying the transformation
\begin{align}\label{pi^2/y}
  \Omega \left( 1-x,e^{-\pi ^2/y},yz/\pi ,\frac{1}{\pi}\left( \frac{1}{x\left( 1-x \right) z}-yz' \right) \right) =0,	
\end{align}
by direct calculations, we obtain
\begin{align}\label{thm5.3-proof-final-one}
S_{2p-1,1}\left( \frac{\pi ^2}{y} \right) =\frac{y^{2p}z^{2p}\left( 1-x \right)}{2^{2p}\pi ^{2p}}\mathrm{Q}_{2p-2}\left( 1-x \right)
\end{align}
and
\begin{align}\label{thm5.3-proof-final-two}
&\text{DS}_{2p,2}\left( \frac{\pi ^2}{y} \right)\nonumber
=-x\left( 1-x \right) \frac{y^{2p+2}z^{2p+2}}{2^{2p}\pi ^{2p+2}}\frac{\mathrm{d}}{\mathrm{d}x}\left\{ \left( 1-x \right) \mathrm{Q}_{2p-2}\left( 1-x \right) \right\}
\nonumber\\
&\quad+2p\left( 1-x \right) ^2x\frac{y^{2p+1}z^{2p+1}}{2^{2p}\pi ^{2p+2}}\left[ \frac{1}{x\left( 1-x \right) z}-yz' \right] \mathrm{Q}_{2p-2}\left( 1-x \right)
\nonumber\\
&=-x\left( 1-x \right) \frac{y^{2p+2}z^{2p+2}}{2^{2p}\pi ^{2p+2}}\frac{\mathrm{d}}{\mathrm{d}x}\left\{ \left( 1-x \right) \mathrm{Q}_{2p-2}\left( 1-x \right) \right\}
\nonumber\\
&\quad+2p\left( 1-x \right) \frac{y^{2p+1}z^{2p}}{2^{2p}\pi ^{2p+2}}\mathrm{Q}_{2p-2}\left( 1-x \right) +2p\left( 1-x \right) ^2x\frac{y^{2p+2}z^{2p+1}z'}{2^{2p}\pi ^{2p+2}}\mathrm{Q}_{2p-2}\left( 1-x \right).
\end{align}
Finally, substituting \eqref{thm5.3-proof-final-one} and \eqref{thm5.3-proof-final-two} into \eqref{thm5.3-proof-begin-1} yields the desired evaluation with a direct calculation.
\end{proof}

Setting $p=2,3,4,5$ in \eqref{barS2p3} with the help of \emph{Mathematica}, we get the following evaluations.
\begin{cor}\emph{(\cite[Cor.1]{XuZhao-2024})} We have
\begin{align*}
\bar{S}_{4,2}(y)&=\frac{1}{8}x(1 - x)z^{5}[z - 4(1 - x)z'],\\
\bar{S}_{6,2}(y)&=\frac{1}{8}x(1 - x)z^{7}[6(1 - x)(2 - x)z'+(2x - 3)z],\\
\bar{S}_{8,2}(y)&=\frac{1}{8}x(1 - x)z^{9}[4(2x^{3}-19x^{2}+34x - 17)z'+(3x^{2}-19x + 17)z],\\
\bar{S}_{10,2}(y)&=\frac{1}{8}x(1-x)z^{11}\left\{ \begin{array}{c}
10(x^4-34x^3+126x^2-155x+62)z'\\
+(4x^3-102x^2+252x-155)z\\
\end{array} \right\} .
\end{align*}
\end{cor}
From \cite[Cor.1]{XuZhao-2024}, we also obtain
$$\bar{S}_{2,2}(y)=\frac{1}{8}x(1-x)z^2\left\{ 4x(1-x)\left( z' \right) ^2+4\left( 1-x \right) z'z-z^2 \right\}.$$

\section{Evaluations of Mixed Berndt-type Integrals}
We assume that $x_0, y_0$ and $z_0$ are all known. For $\Omega \left( x_0,\exp \left( -y_0 \right) ,z_0,z_0' \right) =0$, considering the transformation \eqref{lem-for-trans-3}
$$
\Omega \left( \left( \frac{1-\sqrt{1-x}}{1+\sqrt{1-x}} \right) ^2,\exp \left( -2y \right) ,\frac{z\left( 1+\sqrt{1-x} \right)}{2},\frac{\left( 1+\sqrt{1-x} \right) ^3}{4\left( 1-\sqrt{1-x} \right)}\left( \left( \sqrt{1-x}+1-x \right) z'-\frac{z}{2} \right) \right) =0,
$$
by an elementary calculation, we can deduce its inverse transformation
\begin{align}\label{trsform}
\begin{cases}
x=\frac{4\sqrt{x_0}}{1+x_0+2\sqrt{x_0}},\\
y=\frac{y_0}{2},\\
z=z_0\left( 1+\sqrt{x_0} \right) ,\\
z'=\sqrt{x_0}\frac{\left( 1+\sqrt{x_0} \right) ^5}{1-x_0}\frac{z_0'}{2}+\frac{\left( 1+\sqrt{x_0} \right) ^4}{1-x_0}\frac{z_0}{4}.
\end{cases}
\end{align}

\begin{thm}\label{mixInt1}
Let $m\in \mathbb{N}$ and $\Gamma=\Gamma(1/4)$. Then
\begin{align}
&\int_0^{\infty}{\frac{x^{4m}\left( \sinh x-\sin x \right) \mathrm{d}x}{\left[ \sinh ^2x+\sin ^2x \right] \left[ \cosh x+\cos x \right]}}= c_{1,m} \frac{\Gamma ^{8m}}{\pi ^{2m}}+ c_{2,m}\frac{\Gamma ^{8m+4}}{\pi ^{2m+2}},
\\
&\int_0^{\infty}{\frac{x^{4m-2}\left( \sin x+\sinh x \right) \mathrm{d}x}{\left[ \sinh ^2x+\sin ^2x \right] \left[ \cosh x+\cos x \right]}}= d_{1,m} \frac{\Gamma ^{8m-4}}{\pi ^{2m-1}}+d_{2,m} \frac{\Gamma ^{8m}}{\pi ^{2m+1}},
	\end{align}
	where  numbers $c_{1,m},c_{2,m},d_{1,m},d_{2,m}\in \mathbb{Q}$.
\end{thm}
\begin{proof}
According to  Lemma \ref{Taylor sn},  we know that $\text{q}_{2p-1}(x)=\sum^{p-1}_{j=0}a_jx^j$.
Applying Theorem \ref{barC2p,2} and transformation \eqref{trsform},  one obtains
\begin{align}
&\bar{C}_{2p,2}\left( y\right)=\bar{C}_{2p,2}\left( \frac{y_0}{2} \right)\nonumber\\
& =\left( -1 \right) ^p\frac{px(1-x)}{2^{2p-1}}\mathrm{q}_{2p-1}\left( 1-x \right) \sqrt{1-x}z^{2p+1}z'
	+\left( -1 \right) ^p\frac{x(1-x)}{2^{2p}}\frac{\mathrm{d}}{\mathrm{d}x}\left[ \mathrm{q}_{2p-1}\left( 1-x \right) \sqrt{1-x} \right] z^{2p+2}
		\nonumber\\
		&=
		\left( -1 \right) ^p\frac{p\sqrt{x_0}\left( 1-\sqrt{x_0} \right) ^3}{2^{2p-3}\left( 1+\sqrt{x_0} \right) ^{4-2p}}\mathrm{q}_{2p-1}\left( \frac{\left( 1-\sqrt{x_0} \right) ^2}{\left( 1+\sqrt{x_0} \right) ^2} \right) z_{0}^{2p+1}
\left\{ \sqrt{x_0}\frac{\left( 1+\sqrt{x_0} \right) ^5}{1-x_0}\frac{z_0'}{2}+\frac{\left( 1+\sqrt{x_0} \right) ^4}{1-x_0}\frac{z_0}{4} \right\}
		\nonumber\\
		&\quad -\frac{\left( -1 \right) ^p}{2^{2p-2}}\frac{\sqrt{x_0}\left( 1-\sqrt{x_0} \right) ^2}{\left( 1+\sqrt{x_0} \right) ^4}\sum_{j=0}^{p-1}{a_j}\left( j+\frac{1}{2} \right) \left( \frac{1-\sqrt{x_0}}{1+\sqrt{x_0}} \right) ^{2j-1}\left( 1+\sqrt{x_0} \right) ^{2p+2}z_{0}^{2p+2}
		\nonumber\\
		&=\left( -1 \right) ^p\frac{px_0\left( 1-x_0 \right) ^2}{2^{2p-2}}\left( 1+\sqrt{x_0} \right) ^{2p-2}\mathrm{q}_{2p-1}\left( \left( \frac{1-\sqrt{x_0}}{1+\sqrt{x_0}} \right) ^2 \right) z_{0}^{2p+1}z_0'
		\nonumber\\
		&\quad
		-\frac{\left( -1 \right) ^p}{2^{2p-2}}(1-x_0)^2\sqrt{x_0}\sum_{j=0}^{p-1}{a_j}\left( j+\frac{1}{2} \right) \left( \frac{1-\sqrt{x_0}}{1+\sqrt{x_0}} \right) ^{2j-1}\left( 1+\sqrt{x_0} \right) ^{2p-4}z_{0}^{2p+2}
			\nonumber\\
		&\quad
		+\frac{\left( -1 \right) ^p}{2^{2p-2}}(1-x_0)^2\frac{p}{2}\sqrt{x_0}\mathrm{q}_{2p-1}\left( \frac{\left( 1-\sqrt{x_0} \right) ^2}{\left( 1+\sqrt{x_0} \right) ^2} \right) \left( 1+\sqrt{x_0} \right) ^{2p-3}z_{0}^{2p+2}.
		\label{cbar(y/2)}
	\end{align}
Setting $x_0=1/2$ in \eqref{cbar(y/2)}  yields  $y_0=\pi$ via \eqref{notations-Ramanujan}. All coefficients preceding $z_{0}^{2p+2}$ and $z_{0}^{2p+1}z_0'$ are  rational numbers.  When $x_0=1/2$, the coefficient of $z_{0}^{2p+1}z_0'$ is
	\begin{align}
		\left( 1+\frac{\sqrt{2}}{2} \right) ^{2p-2}\frac{\left( -1 \right) ^pp}{2^{2p+1}}\mathrm{q}_{2p-1}\left( \left( \frac{\sqrt{2}-1}{\sqrt{2}+1} \right) ^2 \right) 	\nonumber\\=\frac{\left( -1 \right) ^pp}{2^{3p}}\left( \sqrt{2}+1 \right) ^{2p-2}\sum_{j=0}^{p-1}{a_j\left( \frac{\sqrt{2}-1}{\sqrt{2}+1} \right) ^{2j}}.
	\end{align}
In fact, Lemma \ref{sn cofficient symmetry} indicates  that $a_j=a_{p-1-j}$ for $ j=0,\ldots,p-1$. From \eqref{even integer}, we can directly conclude that
\begin{align}
		&\left( \sqrt{2}+1 \right) ^{2p-2}\left\{ a_j\left( \frac{\sqrt{2}-1}{\sqrt{2}+1} \right) ^{2j}+a_{p-1-j}\left( \frac{\sqrt{2}-1}{\sqrt{2}+1} \right) ^{2\left( p-1-j \right)} \right\}
		\nonumber\\
		&
		=a_j\left\{ \left( \sqrt{2}-1 \right) ^{4j-2p+2}+\left( \sqrt{2}+1 \right) ^{2p-2}\left( \sqrt{2}-1 \right) ^{4\left( p-1-j \right)} \right\}
		\nonumber\\
		&
		=a_j\left\{ \left( \sqrt{2}-1 \right) ^{4j-2p+2}+\left( \sqrt{2}+1 \right) ^{4j-2p+2} \right\} \in \mathbb{Z}, \quad (j=1,\ldots,[(p-1)/2]).
			\end{align}
where $ [\cdot ]$ is the floor function. Further, for $j=0,\ldots,p-1$, we obtain
\begin{align}\label{Q1}
\frac{\left( -1 \right) ^pp}{2^{3p}}\left( \sqrt{2}+1 \right) ^{2p-2}\sum_{j=0}^{p-1}{a_j\left( \frac{\sqrt{2}-1}{\sqrt{2}+1} \right) ^{2j}}\in \mathbb{Q}.
\end{align}
Similarly, when $x_0=1/2$, the coefficient of $z_{0}^{2p+2}$ is
\begin{align}
&\quad\frac{\left( -1 \right) ^p}{2^{2p+1}}\sqrt{2}\left\{ \begin{array}{c}
\frac{p}{2}\mathrm{q}_{2p-1}\left( \left( \frac{\sqrt{2}-1}{\sqrt{2}+1} \right) ^2 \right) \left( 1+\frac{\sqrt{2}}{2} \right) ^{2p-3}\\
-\sum_{j=0}^{p-1}{a_j}\left( j+\frac{1}{2} \right) \left( \frac{\sqrt{2}-1}{\sqrt{2}+1} \right) ^{2j-1}\left( 1+\frac{\sqrt{2}}{2} \right) ^{2p-4}\\
\end{array} \right\}
\nonumber\\
&\quad \quad =
\frac{\left( -1 \right) ^p}{2^{3p+1}}\sqrt{2}\left( 1+\sqrt{2} \right) ^{2p-4}\sum_{j=0}^{p-1}{a_j}\left\{ \begin{array}{c}
\sqrt{2}p\left( 1+\sqrt{2} \right) \left( \frac{\sqrt{2}-1}{\sqrt{2}+1} \right) ^{2j}\\
-\left(4j+2 \right) \left( \frac{\sqrt{2}-1}{\sqrt{2}+1} \right) ^{2j-1}\\
\end{array} \right\}
\nonumber\\
&\quad \quad=
\frac{\left( -1 \right) ^p}{2^{3p+1}}\sqrt{2}\left( 1+\sqrt{2} \right) ^{2p-4}\sum_{j=0}^{p-1}{a_j}\left\{ \left( 2-\sqrt{2} \right) p-\left( 4j+2 \right) \right\} \left( \sqrt{2}-1 \right) ^{4j-2}.
\end{align}
According to \eqref{odd integer}, \eqref{even integer} and $a_j=a_{p-1-j},j=0,\ldots,p-1$, we have
\begin{align}
&\sqrt{2}\left( 1+\sqrt{2} \right) ^{2p-4}\left\{ \begin{array}{c}
	a_j\left\{ \left( 2-\sqrt{2} \right) p-2\left( 2j+1 \right) \right\} \left( \sqrt{2}-1 \right) ^{4j-2}+\\
	a_{p-1-j}\left\{ \left( 2-\sqrt{2} \right) p-2\left( 2p-2j-1 \right) \right\} \left( \sqrt{2}-1 \right) ^{4p-4j-6}\\
\end{array} \right\} 
\nonumber\\
&
=a_j\left\{ \begin{array}{c}
	\left\{ 2\left( \sqrt{2}-1 \right) p-2\sqrt{2}\left( 2j+1 \right) \right\} \left( \sqrt{2}+1 \right) ^{2p-4j-2}\\
	+\left\{ 2\sqrt{2}\left( 2j+1 \right) -2\left( \sqrt{2}+1 \right) p \right\} \left( \sqrt{2}-1 \right) ^{2p-4j-2}\\
\end{array} \right\} 
\nonumber\\
&
=2a_j\left\{ \begin{array}{c}
	\sqrt{2}\left( 2j+1 \right) \left\{ \left( \sqrt{2}-1 \right) ^{2p-4j-2}-\left( \sqrt{2}+1 \right) ^{2p-4j-2} \right\}\\
	+p\left\{ \left( \sqrt{2}+1 \right) ^{2p-4j-3}-\left( \sqrt{2}-1 \right) ^{2p-4j-3} \right\}\\
\end{array} \right\} \in \mathbb{Z}, 
\label{Q2}
\end{align}
where $j=1,\ldots,[(p-1)/2]$. Combining \eqref{Q1} and \eqref{Q2} gives
\begin{align} \label{barCQ}
\bar{C}_{2p,2}\left( \frac{\pi}{2} \right) \in \mathbb{Q} z_{0}^{2p+1}z_{0}^{\prime}+\mathbb{Q} z_{0}^{2p+2}.
\end{align}
We note the $n$-th derivative of $z$, as given in \eqref{notations-Ramanujan}, is
\begin{align}\label{n-th derivative}
\frac{\mathrm{d}^nz}{\mathrm{d}x^n}=\frac{\left( 1/2 \right) _{n}^{2}}{n!} \,_{2}F_1\left( \frac{1}{2}+n,\frac{1}{2}+n,1+n;x \right).
\end{align}
To evaluate this derivative at $x_0=\frac{1}{2}$, we
apply the known identity (\cite{A2000})
\begin{align}\label{known identity}
\,_{2}F_1\left( a,b,\frac{a+b+1}{2};\frac{1}{2} \right) =\frac{\Gamma \left( \frac{1}{2} \right) \Gamma \left( \frac{a+b+1}{2} \right)}{\Gamma \left( \frac{a+1}{2} \right) \Gamma \left( \frac{b+1}{2} \right)}.
\end{align}
By substituting $a=b=\frac{1}{2}+n$ into \eqref{known identity} and setting \(x_0 = \frac{1}{2}\) in \eqref{n-th derivative}, we obtain
\begin{align}\label{x=1/2,n-th}
\frac{d^nz}{dx^n}\mid_{x=x_0}^{}=\frac{\left( 1/2 \right) _{n}^{2}\sqrt{\pi}}{\Gamma ^2\left( \frac{n}{2}+\frac{3}{4} \right)}.
\end{align} Thus, we evaluate $y_0=\pi, z_0=\frac{\Gamma ^2}{2\pi ^{3/2}}$, $z_0'=\frac{4\pi ^{1/2}}{\Gamma ^2}$ by  \eqref{notations-Ramanujan} and \eqref{x=1/2,n-th} , then substituting them into \eqref{barCQ} yields
\begin{align} \label{bar2p}
\bar{C}_{2p,2}\left( \frac{\pi}{2} \right) \in \mathbb{Q} \frac{\Gamma ^{4p}}{\pi ^{3p+1}}+\mathbb{Q} \frac{\Gamma ^{4p+4}}{\pi ^{3p+3}}.
\end{align}
In Theorem \ref{thm1cosh++}, letting  $a=4m-2$ and $4m$, we have
\begin{align} \label{4m-2,4m}
\int_0^{\infty}{\frac{x^{4m-2}\left( \sinh x+\sin x \right) \mathrm{d}x}{\left[ \sinh ^2x+\sin ^2x \right] \left[ \cosh x+\cos x \right]}}=
\left( -1 \right) ^{m-1}\frac{\pi ^{4m-1}}{4^m}\bar{C}_{4m-2,2}\left( \frac{\pi}{2} \right)
,\\
\int_0^{\infty}{\frac{x^{4m}\left( \sinh x-\sin x \right) \mathrm{d}x}{\left[ \sinh^2x+\sin ^2x \right] \left[ \cosh x+\cos x \right]}}
=\left( -1 \right) ^{m-1}\frac{\pi ^{4m+1}}{2\cdot 4^m}\bar{C}_{4m,2}\left( \frac{\pi}{2} \right) \label{4m-2,4m1}
.
\end{align}
Finally, substituting \eqref{bar2p} into \eqref{4m-2,4m} and \eqref{4m-2,4m1}, respectively, we prove the theorem.
\end{proof}
\begin{re}
In particular, by using the methods of the proof in \cite[Thm 4.1]{Berndt2016}, we consider
	\begin{align*}
\mathscr{Q}=\lim_{R\rightarrow \infty} \int_{D_R}^{}{\frac{\mathrm{d}z}{\left( \sinh z+i\sin z \right) \left( \cosh z+\cos z \right)}},
	\end{align*}
where \( D_R \) is the same contour that we considered above, except that now the contour is indented at the origin by a quarter-circle \( \Gamma_\epsilon \) of radius \( \epsilon \).  

On \( \Gamma_\epsilon \), set \( z = \epsilon e^{i\theta} \), \( \frac{1}{2}\pi \geq \theta \geq 0 \), and  
the contour  is shown in the following figure:
	\begin{center}
		\includegraphics[height=2.5in]{contour3}
	\end{center}
It can be obtained by the residue method that
	\begin{align}\label{m=0residue method}
	\int_0^{\infty}{\frac{\left( \sinh x-\sin x \right) \mathrm{d}x}{\left[ \sinh ^2x+\sin ^2x \right] \left[ \cosh x+\cos x \right]}}=\frac{\pi}{8}-\frac{\pi}{2}\bar{C}_{0,2}\left( \frac{\pi}{2} \right),
	\end{align}
and by using Weierstrass elliptic functions \cite{L1975,L1988}  to obtain closed forms of $\bar{C}_{0,2}\left( \frac{\pi}{2} \right)$
	\begin{align}\label{barC0,2(pi/2)}
\bar{C}_{0,2}\left( \frac{\pi}{2} \right) =\frac{1}{2}-\frac{\Gamma ^4}{4\pi ^5}.
	\end{align}
Substituting \eqref{barC0,2(pi/2)} into \eqref{m=0residue method} ,
 we can find that
	\begin{align}
\int_0^{\infty}{\frac{\left( \sinh x-\sin x \right) \mathrm{d}x}{\left[ \sinh ^2x+\sin ^2x \right] \left[ \cosh x+\cos x \right]}}=\frac{\Gamma ^4}{32\pi ^2}-\frac{\pi}{8}.
	\end{align}
\end{re}

By \emph{Mathematica} we can compute the following easily using the formulas in Theorem \ref{mixInt1}:
\begin{exa}\label{ExampleC}
Let $\Gamma=\Gamma(1/4)$. We have
	\begin{align}
	\int_0^{\infty}{\frac{x^2\left( \sinh x+\sin x \right) \mathrm{d}x}{\left[ \sinh ^2x+\sin ^2x \right] \left[ \cosh x+\cos x \right]}}&=\frac{\Gamma ^8}{2^9\pi ^3}-\frac{\Gamma ^4}{2^6\pi},
	\nonumber\\
	\int_0^{\infty}{\frac{x^6\left( \sinh x+\sin x \right) \mathrm{d}x}{\left[ \sinh ^2x+\sin ^2x \right] \left[ \cosh x+\cos x \right]}}&=\frac{9\Gamma  ^{12}}{16384\pi ^3}-\frac{3\Gamma  ^{16}}{131072\pi ^5},
	\nonumber\\
	\int_0^{\infty}{\frac{x^{10}\left( \sinh x+\sin x \right) \mathrm{d}x}{\left[ \sinh ^2x+\sin ^2x \right] \left[ \cosh x+\cos x \right]}}&=\frac{459\Gamma  ^{24}}{33554432\pi ^7}-\frac{945\Gamma  ^{20}}{4194304\pi ^5},
	\nonumber\\
	\int_0^{\infty}{\frac{x^4\left( \sinh x-\sin x \right) \mathrm{d}x}{\left[ \sinh ^2x+\sin ^2x \right] \left[ \cosh x+\cos x \right]}}&=\frac{3\Gamma  ^8}{1024\pi ^2}-\frac{\Gamma ^{12}}{8192\pi ^4},
	\nonumber\\
	\int_0^{\infty}{\frac{x^{8}\left( \sinh x-\sin x \right) \mathrm{d}x}{\left[ \sinh ^2x+\sin ^2x \right] \left[ \cosh x+\cos x \right]}}&=\frac{33\Gamma  ^{20}}{2097152\pi ^6}-\frac{63\Gamma  ^{16}}{262144\pi ^4}.	\nonumber\
	\end{align}
\end{exa}

\begin{thm}\label{mixInt2}
Let $m\in \mathbb{N} $ and $\Gamma=\Gamma(1/4)$. Then
\begin{align}\label{IntbarS2p}
\int_0^{\infty}{\frac{x^{4m}\left( \sinh x+\sin x \right) \mathrm{d}x}{\left[ \sinh ^2x+\sin ^2x \right] \left[ \cosh x-\cos x \right]}}&= e_{1,m} \frac{\Gamma ^{8m}}{\pi ^{2m}}+e_{2,m} \frac{\Gamma ^{8m+4}}{\pi ^{2m+2}},
\\
\int_0^{\infty}{\frac{x^{4m-2}\left( \sinh x-\sin x \right) \mathrm{d}x}{\left[ \sinh ^2x+\sin ^2x \right] \left[ \cosh x-\cos x \right]}}&= f_{1,m} \frac{\Gamma ^{8m-4}}{\pi ^{2m-1}}+f_{2,m} \frac{\Gamma ^{8m}}{\pi ^{2m+1}}\quad (m\ge 2).
\end{align}
where the numbers $e_{1,m}, e_{2,m} ,f_{1,m}, f_{2,m} \in \mathbb{Q}$.
\end{thm}
\begin{proof}
Applying Theorem \ref{barS2p,2} and \eqref{trsform}, by a direct calculation, we deduce
\begin{align}
	&\bar{S}_{2p,2}\left( y\right)=\bar{S}_{2p,2}\left( \frac{y_0}{2} \right)\\& =\frac{\left( -1 \right) ^{p-1}x\left( 1-x \right)}{2^{2p}}\left[ \begin{array}{c}
		z^{2p+2}\frac{\mathrm{d}}{\mathrm{d}x}\left\{ \left( 1-x \right) \mathrm{Q}_{2p-2}\left( 1-x \right) \right\}\\
		+2p\left( 1-x \right) \mathrm{Q}_{2p-2}\left( 1-x \right) z^{2p+1}z'\\
	\end{array} \right]
	\nonumber\\
	&=\frac{\left( -1 \right) ^{p-1}\frac{4\sqrt{x_0}}{1+x_0+2\sqrt{x_0}}\frac{1+x_0-2\sqrt{x_0}}{1+x_0+2\sqrt{x_0}}}{2^{2p}}{z_0}^{2p+1}\left( 1+\sqrt{x_0} \right) ^{2p+1}\nonumber\\&\quad\quad\times\left[ \begin{array}{c}
		z_0\left( 1+\sqrt{x_0} \right) 	\frac{\mathrm{d}}{\mathrm{d}x}\left\{ \left( 1-x \right) \mathrm{Q}_{2p-2}\left( 1-x \right) \right\} \mid_{4\sqrt{x_0}/\left( 1+\sqrt{x_0} \right) ^2}\\
		+2p\frac{1+x_0-2\sqrt{x_0}}{1+x_0+2\sqrt{x_0}}\mathrm{Q}_{2p-2}\left( \frac{1+x_0-2\sqrt{x_0}}{1+x_0+2\sqrt{x_0}} \right)
		\left\{ \sqrt{x_0}\frac{\left( 1+\sqrt{x_0} \right) ^5}{1-x_0}\frac{z_0'}{2}+\frac{\left( 1+\sqrt{x_0} \right) ^4}{1-x_0}\frac{z_0}{4} \right\}\\
	\end{array} \right] 	\nonumber\\
	&
	=\frac{\left( -1 \right) ^{p-1}\left( 1-x_0 \right) ^4{z_0}^{2p+2}}{2^{2p-3}}p\sqrt{x_0}\left( 1+\sqrt{x_0} \right) ^{2p-9}\mathrm{Q}_{2p-2}\left(
	\left( \frac{1-\sqrt{x_0}}{1+\sqrt{x_0}} \right) ^2
	 \right)\nonumber\\&\quad\quad\times \left\{ \sqrt{x_0}\frac{\left( 1+\sqrt{x_0} \right) ^5}{1-x_0}\frac{z_0'}{2}+\frac{\left( 1+\sqrt{x_0} \right) ^4}{1-x_0}\frac{z_0}{4} \right\}+\frac{\left( -1 \right) ^{p-1}\left( 1-x_0 \right) ^2{z_0}^{2p+2}}{2^{2p-2}}\sqrt{x_0}\left( 1+\sqrt{x_0} \right) ^{2p-4}
	\nonumber\\
	&\quad\quad\times
	\frac{\mathrm{d}}{\mathrm{d}x}\left\{ \left( 1-x \right) \mathrm{Q}_{2p-2}\left( 1-x \right) \right\} \mid_{4\sqrt{x_0}/\left( 1+\sqrt{x_0} \right) ^2}.
\label{Sbar{y/2}}
	\end{align}
When $x_0=1/2$ in \eqref{Sbar{y/2}}, we prove that the coefficients before $z_{0}^{2p+2},z_{0}^{2p+1}z_0'$ are all rational numbers.
The coefficient before $z_0'{z_0}^{2p+1}$ is
\begin{align}
	\frac{\left( -1 \right) ^{p-1}\left( 1-x_0 \right) ^2}{2^{2p-1}}&\sqrt{x_0}\left( 1+\sqrt{x_0} \right) ^{2p-5}\sqrt{x_0}\frac{\left( 1+\sqrt{x_0} \right) ^5}{1-x_0}2p\frac{1+x-2\sqrt{x_0}}{1+x_0+2\sqrt{x_0}}\mathrm{Q}_{2p-2}\left( \frac{1+x-2\sqrt{x_0}}{1+x_0+2\sqrt{x_0}} \right)
	\nonumber\\
	&=\frac{\left( -1 \right) ^{p-1}p\left( 1-x_0 \right) ^3x_0}{2^{2p-2}}\left( 1+\sqrt{x_0} \right) ^{2p-4}\mathrm{Q}_{2p-2}\left( \left( \frac{1-\sqrt{x_0}}{1+\sqrt{x_0}} \right) ^2 \right)
	\nonumber\\
	&=\frac{\left( -1 \right) ^{p-1}p\left( 1-x_0 \right) ^3x_0}{2^{2p-2}}\left( 1+\sqrt{x_0} \right) ^{2p-4}\sum_{j=0}^{p-2}{b_j\left( \frac{1-\sqrt{x_0}}{1+\sqrt{x_0}} \right) ^{2j}}.
\end{align}
Applying $a_j=a_{p-1-j}$ and $\mathrm{Q}_{2p-2}\left( x \right) =\sum_{j=0}^{p-2}{\left( \begin{array}{c}
			2p-2\\
			2j+1\\
		\end{array} \right) \mathrm{q}_{2j+1}\left( x \right) \mathrm{q}_{2p-2j-3}\left( x \right)}$,
 we have
 \begin{align}\label{Q2p-2}
	\mathrm{Q}_{2p-2}\left( x \right) =\sum_{j=0}^{p-2}{b_jx^j}\quad (b_j\in \Z),
		\end{align} wich can reaily reduce that $b_j=b_{p-j-2}\ (j=0,\ldots,p-2)$. Thus,
\begin{align}
\left( \sqrt{2}+1 \right) ^{2p-4}&\left\{ b_j\left( \frac{\sqrt{2}-1}{\sqrt{2}+1} \right) ^{2j}+b_{p-j-2}\left( \frac{\sqrt{2}-1}{\sqrt{2}+1} \right) ^{2\left( p-j-2 \right)} \right\}
	\nonumber\\
&=b_j\left\{ \left( \sqrt{2}-1 \right) ^{4j-2p+4}+\left( \sqrt{2}+1 \right) ^{2p-4}\left( \sqrt{2}-1 \right) ^{4\left( p-j-2 \right)} \right\}
	\nonumber\\
&=b_j\left\{ \left( \sqrt{2}-1 \right) ^{4j-2p+4}+\left( \sqrt{2}+1 \right) ^{4j-2p+4} \right\}\in \Z \quad (j=0,\ldots ,p-2).
	\end{align}
Hence, we obtain
\begin{align} \label{Q12}
	\frac{\left( -1 \right) ^{p-1}p}{2^{3p}}\left( \sqrt{2}+1 \right) ^{2p-4}\sum_{j=0}^{p-2}{b_j\left( \frac{\sqrt{2}-1}{\sqrt{2}+1} \right) ^{2j}}\in \mathbb{Q} .
	 \end{align}
Similarly, the coefficient of ${z_0}^{2p+2}$ is
\begin{align}\label{z^(2p+2)}
	\frac{\left( -1 \right) ^{p-1}\left( 1-x_0 \right) ^2}{2^{2p-2}}\sqrt{x_0}\left( 1+\sqrt{x_0} \right) ^{2p-5}\left[ \begin{array}{c}
		\left( 1+\sqrt{x_0} \right) \frac{\mathrm{d}}{\mathrm{d}x}\left\{ \left( 1-x \right) \mathrm{Q}_{2p-2}\left( 1-x \right) \right\}\\
		+\frac{p}{2}\left( 1-x_0 \right) \mathrm{Q}_{2p-2}\left( \left( \frac{1-\sqrt{x_0}}{1+\sqrt{x_0}} \right) ^2 \right)\\
	\end{array} \right] .
 \end{align}
Applying \eqref{Q2p-2} to \eqref{z^(2p+2)}, we have
\begin{align}
&\frac{\left( -1 \right) ^{p-1}\left( 1-x_0 \right) ^2}{2^{2p-2}}\sqrt{x_0}\left( 1+\sqrt{x_0} \right) ^{2p-5}\sum_{j=0}^{p-2}{b_j}\left[ \frac{p}{2}\left( 1-x_0 \right) -\left( 1+\sqrt{x_0} \right) \left( j+1 \right) \right] \left( \sqrt{2}-1 \right) ^{4j}
	\nonumber\\
&=\frac{\left( -1 \right) ^{p-1}\left( 1-x_0 \right) ^2}{2^{2p-2}}\sqrt{x_0}\left( 1+\sqrt{x_0} \right) ^{2p-4}\sum_{j=0}^{p-2}{b_j}\left[ \frac{p}{2}\left( 1-\sqrt{x_0} \right) -\left( j+1 \right) \right] \left( \sqrt{2}-1 \right) ^{4j}
	\nonumber\\
&=
\frac{\left( -1 \right) ^{p-1}}{2^{3p-2}}\left( \sqrt{2}+1 \right) ^{2p-4}\sum_{j=0}^{p-2}{b_j}\left[ \frac{p}{4}\left( \sqrt{2}-1 \right) -\frac{\sqrt{2}}{2}\left( j+1 \right) \right] \left( \sqrt{2}-1 \right) ^{4j}.
\end{align}
Noting the fact that
\begin{align} \label{Q22}
&\left( \sqrt{2}+1 \right) ^{2p-4}\left\{ \begin{array}{c}
	\left[ \frac{p}{4}\left( \sqrt{2}-1 \right) -\frac{\sqrt{2}}{2}\left( j+1 \right) \right] \left( \sqrt{2}-1 \right) ^{4j}+\\
	\left[ \frac{p}{4}\left( \sqrt{2}-1 \right) -\frac{\sqrt{2}}{2}\left( p-j-1 \right) \right] \left( \sqrt{2}-1 \right) ^{4\left( p-j-2 \right)}\\
\end{array} \right\}
	\nonumber\\
&=\left[ \frac{p}{4}\left( \sqrt{2}-1 \right) -\frac{\sqrt{2}}{2}\left( j+1 \right) \right] \left( \sqrt{2}+1 \right) ^{2p-4j-4}+\left[ -\frac{p}{4}\left( \sqrt{2}+1 \right) +\frac{\sqrt{2}}{2}\left( j+1 \right) \right] \left( \sqrt{2}-1 \right) ^{2p-4j-4}
	\nonumber\\
&=\frac{p}{4}\left[ \left( \sqrt{2}+1 \right) ^{2p-4j-5}-\left( \sqrt{2}-1 \right) ^{2p-4j-5} \right]\nonumber\\& \quad +\frac{\sqrt{2}}{2}\left( j+1 \right) \left[ \left( \sqrt{2}-1 \right) ^{2p-4j-4}-\left( \sqrt{2}+1 \right) ^{2p-4j-4} \right]\in \Q ,
\end{align}
where $j=0,\ldots ,[(p-2)/2]$.
Combining \eqref{Q12} with \eqref{Q22}, we derive that
\begin{align} \label{barSQ}
\bar{S}_{2p,2}\left( \frac{\pi}{2} \right) \in \mathbb{Q} z_{0}^{2p+1}z_{0}^{\prime}+\mathbb{Q} z_{0}^{2p+2}\quad (p\ge 2).
\end{align}
Substituting $z_0=\frac{\Gamma ^2}{2\pi ^{3/2}}$ and $z_0'=\frac{4\pi ^{1/2}}{\Gamma ^2}$ into \eqref{barSQ} gives
\begin{align} \label{barS2pQ}
		\bar{S}_{2p,2}\left( \frac{\pi}{2} \right) \in \mathbb{Q} \frac{\Gamma ^{4p}}{\pi ^{3p+1}}+\mathbb{Q} \frac{\Gamma ^{4p+4}}{\pi ^{3p+3}}.
	\end{align}
In Theorem \ref{Sin+,Cos-}, letting $a=4m$ and $4m-2$, we have
\begin{align}
\int_0^{\infty}{\frac{x^{4m}\left( \sinh x+\sin x \right) \mathrm{d}x}{\left[ \sinh ^2x+\sin ^2x \right] \left[ \cosh x-\cos x \right]}}&=\left( -1 \right) ^{m-1}\frac{\pi ^{4m+1}}{2\cdot 4^m}\bar{S}_{4m,2}\left( \frac{\pi}{2} \right) ,\label{4m,S}
\\
	\int_0^{\infty}{\frac{x^{4m-2}\left( \sinh x-\sin x \right) \mathrm{d}x}{\left[ \sinh ^2x+\sin ^2x \right] \left[ \cosh x-\cos x \right]}}&=\left( -1 \right) ^{m-1}\frac{\pi ^{4m-1}}{4^m}\bar{S}_{4m-2,2}\left( \frac{\pi}{2} \right) \label{4m-2,S}.
\end{align}
Finally, substituting \eqref{barS2pQ} into \eqref{4m,S} and \eqref{4m-2,S} yields the desired result.
\end{proof}
\begin{re}
In particular, by using the methods of the proof in Theorem \ref{mixInt2} and the explicit expression of $\bar{S}_{2,2}\left( \frac{\pi}{2} \right)$, we can find that
\begin{align}
	\int_0^{\infty}{\frac{x^2\left( \sinh x-\sin x \right) \mathrm{d}x}{\left[ \sinh ^2x+\sin ^2x \right] \left[ \cosh x-\cos x \right]}}=\frac{\Gamma ^4}{2^6\pi}-\frac{\pi}{8}.
\end{align}
\end{re}

By \emph{Mathematica} we can compute the following easily using the formulas in Theorem \ref{mixInt2}:
\begin{exa}\label{ExampleS}
		Let $\Gamma=\Gamma(1/4)$. Then
		\begin{align}
\int_0^{\infty}{\frac{x^4\left( \sinh x+\sin x \right) \mathrm{d}x}{\left[ \sinh ^2x+\sin ^2x \right] \left[ \cosh x-\cos x \right]}}&=\frac{\Gamma ^{12}}{2^{13}\pi ^4}-\frac{\Gamma ^8}{2^{10}\pi ^2},
	\nonumber\\
\int_0^{\infty}{\frac{x^8\left( \sinh x+\sin x \right) \mathrm{d}x}{\left[ \sinh ^2x+\sin ^2x \right] \left[ \cosh x-\cos x \right]}}&=\frac{81\Gamma ^{16}}{262144\pi ^4}-\frac{33\Gamma ^{20}}{2097152\pi ^6},
	\nonumber\\
\int_0^{\infty}{\frac{x^6\left( \sinh x-\sin x \right) \mathrm{d}x}{\left[ \sinh ^2x+\sin ^2x \right] \left[ \cosh x-\cos x \right]}}&=\frac{5\Gamma ^{16}}{131072\pi ^5}-\frac{9\Gamma ^{12}}{16384\pi ^3},
	\nonumber\\
\int_0^{\infty}{\frac{x^{10}\left( \sinh x-\sin x \right) \mathrm{d}x}{\left[ \sinh ^2x+\sin ^2x \right] \left[ \cosh x-\cos x \right]}}&=\frac{945\Gamma ^{20}}{4194304\pi ^5}-\frac{405\Gamma ^{24}}{33554432\pi ^7}.	\nonumber
	\end{align}
\end{exa}
\begin{cor}
	For $m\in \mathbb{N}$, we have
	\begin{align}
				\pi \bar{S}_{4m-2,2}\left( \frac{\pi}{2} \right) &=4^{2m}\left[ \left( 2m-1 \right) S_{4m-3,1}\left( 2\pi \right) -\pi \mathrm{DS}_{4m-2,2}\left( 2\pi \right) \right] ,
		\\
		\pi \bar{S}_{4m,2}\left( \frac{\pi}{2} \right) &=4^{2m+1}\left[ \pi \mathrm{DS}_{4m,2}\left( 2\pi \right) -2mS_{4m-1,1}\left( 2\pi \right) \right] ,
		\\
		\pi \bar{C}_{4m-2,2}\left( \frac{\pi}{2} \right) &=4\left[ \pi \mathrm{DS}_{4m-2,2}^{\prime}\left( 2\pi \right) -\left( 4m-2 \right) S_{4m-3,1}^{\prime}\left( 2\pi \right) \right] ,
		\\
		\pi \bar{C}_{4m,2}\left( \frac{\pi}{2} \right) &=4\left[ 4mS_{4m-1,1}^{\prime}\left( 2\pi \right) -\pi \mathrm{DS}_{4m,2}^{\prime}\left( 2\pi \right) \right] .
	\end{align}
\end{cor}
\begin{proof}
	The facts derived from Theorem \ref{Sin-,Cos-} and Theorem \ref{Sin-,Cos+} lead us to the conclusion that
	\begin{align} \label{I-S}
		&\int_0^{\infty}{\frac{x^{4m-2}\left[ \sinh x-\sin x \right] \mathrm{d}x}{\left[ \sinh ^2x+\sin ^2x \right] \left[ \cosh x-\cos x \right]}}=\left( -4 \right) ^m\pi ^{4m-2}\left[ \pi \mathrm{DS}_{4m-2,2}\left( 2\pi \right) -\left( 2m-1 \right) S_{4m-3,1}\left( 2\pi \right) \right] ,
		\nonumber\\
		&\int_0^{\infty}{\frac{x^{4m}\left[ \sinh x+\sin x \right] \mathrm{d}x}{\left[ \sinh ^2x+\sin ^2x \right] \left[ \cosh x-\cos x \right]}}=2\left( -4 \right) ^m\pi ^{4m}\left[ 2mS_{4m-1,1}\left( 2\pi \right) -\pi \mathrm{DS}_{4m,2}\left( 2\pi \right) \right] ,
		\nonumber\\
		&\int_0^{\infty}{\frac{x^{4m-2}\left( \sinh x+\sin x \right) \mathrm{d}x}{\left[ \sinh ^2x+\sin ^2x \right] \left[ \cosh x+\cos x \right]}}=\frac{\left( -1 \right) ^m}{2^{2m-2}}\pi ^{4m-2}\left[ \left( 4m-2 \right) S_{4m-3,1}^{\prime}\left( 2\pi \right) -\pi \mathrm{DS}_{4m-2,2}^{\prime}\left( 2\pi \right) \right] ,
		\nonumber\\
		&\int_0^{\infty}{\frac{x^{4m}\left( \sinh x-\sin x \right) \mathrm{d}x}{\left[ \sinh ^2x+\sin ^2x \right] \left[ \cosh x+\cos x \right]}}=\frac{\left( -1 \right) ^{m-1}\pi ^{4m}}{2^{2m-1}}\left[ 4mS_{4m-1,1}^{\prime}\left( 2\pi \right) -\pi \mathrm{DS}_{4m,2}^{\prime}\left( 2\pi \right) \right] .
	\end{align}
	Hence, combining
	\eqref{4m-2,4m}, \eqref{4m-2,4m1}, \eqref{4m,S}, \eqref{4m-2,S} and \eqref{I-S}, the desired result is established.
	\end{proof}

Finally, we conclude this section by the following conjecture which is supported by
comparing the coefficients in Examples \ref{ExampleC} and \ref{ExampleS}:
\begin{con} For $m\ge 1$, let $c_{2,m},e_{2,m},d_{1,m},f_{1,m}$  be the numbers appearing in
	Theorems \ref{mixInt1} and \ref{mixInt2}.	 Then the following relations hold
$$
c_{2,m}+e_{2,m}=0\quad \text{and}\quad d_{1,m}+f_{1,m}=0.
$$
\end{con}

Next, we evaluate mixed Berndt-type integrals using the generalized Barnes multiple zeta function. Bradshaw and Vignat \cite{BV2024} evaluated generalized Berndt-type integrals \eqref{BTI-definition-1} using the (alternating) Barnes' multiple zeta function. For example, they found that (\cite[Prop. 2]{BV2024})
\begin{align}
&\int_0^\infty \frac{x^{a}dx}{(\cos x-\cosh x)^b}=2^b \Gamma(a+1)\ze_{2b}(a+1,b|(1+i,1-i)^b)\quad (a\geq 2b,b\geq 1)\label{BMZF-RTI-}
\end{align}
and
\begin{align}
&\int_0^\infty \frac{x^{a}dx}{(\cos x+\cosh x)^b}=2^b \Gamma(a+1){\bar \ze}_{2b}(a+1,b|(1+i,1-i)^b)\quad (a\geq 0,b\geq 1)\label{BMZF-RTI+},
\end{align}
where $\bfs^n$ means the string $\bfs$ is repeated $n$ times. Obviously, combining this with the results of Xu and Zhao \cite{XZ2024} yields closed-form evaluations of the Barnes multiple zeta function, see \cite[Cor. 2]{BV2024}.
Here for positive real numbers $a_1,\ldots,a_N$, the \emph{Barnes multiple zeta function} and \emph{alternating Barnes multiple zeta function} are defined by \cite{AIK2014,Ba1904,BV2024,KMT2023} (${\bf a}_N:=\{a_1,\ldots,a_N\}$)
\begin{align}
\ze_N(s,\omega|{\bf a}_N):=\sum_{n_1\geq0,\ldots,n_N\geq 0} \frac{1}{(\omega+n_1a_1+\cdots+n_Na_N)^s}\quad (\Re(\omega)>0,\Re(s)>N)\label{defn:Barneszetafunction}
\end{align}
and
\begin{align}
{\bar \ze}_N(s,\omega|{\bf a}_N):=\sum_{n_1\geq0,\ldots,n_N\geq 0} \frac{(-1)^{n_1+\cdots+n_N}}{(\omega+n_1a_1+\cdots+n_Na_N)^s}\quad (\Re(\omega)>0,\Re(s)> N-1),\label{defn:ABarneszetafunction}
\end{align}
respectively.

In particular, Z.P. Bradshaw and C. Vignat \cite{BV2024} gave the integral representation of Barnes multiple zeta function, which
appears as \cite[Eq. (3.2)]{R2000}.
Here we define the \emph{generalized Barnes multiple zeta function} by
 \begin{align}
 \zeta _N(s,\omega |\mathbf{a}_N;\boldsymbol{\sigma }_N):=\sum_{n_1\ge 0,...,n_N\ge 0}{\frac{\left( \sigma _1 \right) ^{n_1}\cdots \left( \sigma _N \right) ^{n_N}}{(\omega +n_1a_1+\cdots +n_Na_N)^s}}\quad (\Re (\omega )>\Re (s)>N),\quad
\end{align}
where $\boldsymbol{\sigma }_N=\left( \sigma _1,\ldots ,\sigma _N \right) \in \left\{ \pm 1 \right\} ^N$,
and find its an integral representation:

\begin{pro}
 Let $\Re(s)>N, \Re(w)>0$, and $\Re(a_j)>0$ for $j = 1,\ldots,N$. Then
\begin{align*}
\zeta _N(s,\omega |\mathbf{a}_N;\boldsymbol{\sigma }_N)=\frac{1}{\Gamma \left( s \right)}\int_0^{\infty}{u^{s-1}}e^{-wu}\prod_{j=1}^N{(}1-\sigma _je^{-a_ju})^{-1}\mathrm{d}u,
\end{align*}
where $\boldsymbol{\sigma }_N=\left( \sigma _1,\ldots ,\sigma _N \right) \in \left\{ \pm 1 \right\} ^N$.
\end{pro}
\begin{proof}
This proof is analogous to that in \emph{\cite[Prop. 2]{BV2024}}.
\end{proof}

\begin{thm}\label{GBZeta} For positive integer $m$, we have
\begin{align*}
\Gamma \left( 4m+1 \right)\sum_{j=1}^4{i  ^{j-2}\zeta _4\left( 4m+1,3+i^j|\mathbf{a}_4;\boldsymbol{\sigma }_4 \right)}&\in \mathbb{Q} \frac{\Gamma ^{8m}}{\pi ^{2m}}+\mathbb{Q} \frac{\Gamma ^{8m+4}}{\pi ^{2m+2}},
    			\\
\Gamma \left( 4m-1 \right) \sum_{j=1}^4{\left( -i \right) ^{j-2}}\zeta _4\left( 4m-1,3+i^j|\mathbf{a}_4,\boldsymbol{\sigma }_4 \right) &\in \mathbb{Q} \frac{\Gamma ^{8m-4}}{\pi ^{2m-1}}+\mathbb{Q} \frac{\Gamma ^{8m}}{\pi ^{2m+1}},
    			\\
\Gamma \left( 4m+1 \right) \sum_{j=1}^4{\left( -i \right) ^{j-2}}\zeta _4(4m+1,3+i^j|\mathbf{a}_4)&\in \mathbb{Q} \frac{\Gamma ^{8m}}{\pi ^{2m}}+\mathbb{Q} \frac{\Gamma ^{8m+4}}{\pi ^{2m+2}},
    			\\
\Gamma \left( 4m-1 \right)\sum_{j=1}^4{i  ^{j-2}\zeta _4\left( 4m-1,3+i^j|\mathbf{a}_4 \right)}&\in \mathbb{Q} \frac{\Gamma ^{8m-4}}{\pi ^{2m}}+\mathbb{Q} \frac{\Gamma ^{8m}}{\pi ^{2m+1}}\quad (m\ge 2),
\end{align*}
where $\mathbf{a}_4=\left( 2-2i,2+2i,1-i,1+i \right) $ and $\boldsymbol{\sigma }_4=\left( \left\{ 1 \right\}^2,\left\{ -1 \right\}^2 \right)$.
\end{thm}
\begin{proof}
Noting the facts that 	
\begin{align*}
&\frac{\left( \sinh x\pm \sin x \right)}{\left[ \sinh ^2x+\sin ^2x \right] \left[ \cosh x-\cos x \right]}=\frac{e^x-e^{-x}\mp ie^{ix}\pm ie^{-ix}}{4\sinh \left( \left( 1-i \right) x \right) \sinh \left( \left( 1+i \right) x \right) \sinh \left( \frac{1-i}{2}x \right) \sinh \left( \frac{1+i}{2}x \right)},
\\
&\frac{\left( \sinh x\pm \sin x \right)}{\left[ \sinh ^2x+\sin ^2x \right] \left[ \cosh x+\cos x \right]}=\frac{e^x-e^{-x}\mp ie^{ix}\pm ie^{-ix}}{4\sinh \left( \left( 1-i \right) x \right) \sinh \left( \left( 1+i \right) x \right) \cosh \left( \frac{1-i}{2}x \right) \cosh \left( \frac{1+i}{2}x \right)}.
\end{align*}
From  the proof of\cite[Prop. 2]{BV2024}, we have
 \begin{align*}
\int_0^{\infty}{\frac{x^{s-1}e^{-\omega x}}{\prod_{i=1}^M{\sinh \left( a_1x \right)}\prod_{j=1}^N{\cosh \left( b_jx \right)}}}\mathrm{d}x=2^{M+N}\Gamma \left( s \right) \zeta _{M+N}\left( s,w+\sum_{i=1}^M{a_i+}\sum_{i=1}^N{b_i}|\boldsymbol{c}_{M+N},\boldsymbol{\sigma }_{M+N} \right),
\end{align*}
where $
\boldsymbol{c}_{M+N}=\left( 2a_1,\cdots ,2a_M,2b_1,\cdots ,2b_N \right) , \boldsymbol{\sigma }_{M+N}=\left( \left\{ 1 \right\}^M,\left\{ -1 \right\}^N \right) .
$ Thus,
we can easily deduce that
 \begin{align*}
    	\int_0^{\infty}{\frac{x^{4m-2}\left( \sin x+\sinh x \right) \mathrm{d}x}{\left[ \sinh ^2x+\sin ^2x \right] \left[ \cosh x+\cos x \right]}}
    	&=4\Gamma \left( 4m-1 \right) \sum_{j=1}^4{\left( -i \right) ^{j-2}}\zeta _4\left( 4m-1,3+i^j|\mathbf{a}_4,\boldsymbol{\sigma }_4 \right) ,\\
    \int_0^{\infty}{\frac{x^{4m}\left( \sinh x-\sin x \right) \mathrm{d}x}{\left[ \sinh ^2x+\sin ^2x \right] \left[ \cosh x+\cos x \right]}}
  &=4\Gamma \left( 4m+1 \right) \sum_{j=1}^4{ i  ^{j-2}}\zeta _4\left( 4m+1,3+i^j|\mathbf{a}_4,\boldsymbol{\sigma }_4 \right) ,
    	\\
      	\int_0^{\infty}{\frac{x^{4m}\left( \sinh x+\sin x \right) \mathrm{d}x}{\left[ \sinh ^2x+\sin ^2x \right] \left[ \cosh x-\cos x \right]}}
  &	=4\Gamma \left( 4m+1 \right) \sum_{j=1}^4{\left( -i \right) ^{j-2}}\zeta _4(4m+1,3+i^j|\mathbf{a}_4),
    			\end{align*}
and
    	 \begin{align*}
    		\int_0^{\infty}{\frac{x^{4m-2}\left( \sinh x-\sin x \right) \mathrm{d}x}{\left[ \sinh ^2x+\sin ^2x \right] \left[ \cosh x-\cos x \right]}}	=4\Gamma \left( 4m-1 \right) \sum_{j=1}^4{i  ^{j-2}}\zeta _4(4m+1,3+i^j|\mathbf{a}_4),
    	\end{align*}
where $\mathbf{a}_4=\left( 2-2i,2+2i,1-i,1+i \right) $ and $\boldsymbol{\sigma }_4=\left( \left\{ 1 \right\}^2,\left\{ -1 \right\}^2 \right)$.
Finally, using Theorems \ref{mixInt1} and \ref{mixInt2} yields the desired evaluations.
\end{proof}
Using the formulas in Theorem \ref{GBZeta}, we can easily compute the following for \(m = 1, 2\) using \emph{Mathematica}:
\begin{exa}
	Let $\Gamma=\Gamma(1/4)$. Then
\begin{align*}
	&i\zeta _4\left( 5,3-i|\mathbf{a}_4;\boldsymbol{\sigma }_4 \right) -i\zeta _4\left( 5,3+i|\mathbf{a}_4;\boldsymbol{\sigma }_4 \right) 	+\zeta _4\left( 5,2|\mathbf{a}_4;\boldsymbol{\sigma }_4 \right) -\zeta _4\left( 5,4|\mathbf{a}_4;\boldsymbol{\sigma }_4 \right)\\& =
	\frac{\Gamma ^8}{32768\pi ^2}-\frac{\Gamma ^{12}}{786432\pi ^4},
\\
	&i\zeta _4\left( 9,3-i|\mathbf{a}_4;\boldsymbol{\sigma }_4 \right)-i\zeta _4\left( 9,3+i|\mathbf{a}_4;\boldsymbol{\sigma }_4 \right) 	+\zeta _4\left( 9,2|\mathbf{a}_4;\boldsymbol{\sigma }_4 \right) -\zeta _4\left( 9,4|\mathbf{a}_4;\boldsymbol{\sigma }_4 \right) \\&=
\frac{11\Gamma ^{20}}{112742891520\pi ^6}-\frac{\Gamma ^{16}}{671088640\pi ^4},
	\\&
	i\zeta _4\left( 3,3+i|\mathbf{a}_4;\boldsymbol{\sigma }_4 \right)-i\zeta _4\left( 3,3-i|\mathbf{a}_4;\boldsymbol{\sigma }_4 \right) 
	+\zeta _4\left( 3,2|\mathbf{a}_4;\boldsymbol{\sigma }_4 \right) -\zeta _4\left( 3,4|\mathbf{a}_4;\boldsymbol{\sigma }_4 \right)\\& =\frac{\Gamma ^8}{2^{12}\pi ^3}-\frac{\Gamma ^4}{2^9\pi},
	\\
&	i\zeta _4\left( 7,3+i|\mathbf{a}_4;\boldsymbol{\sigma }_4 \right)-i\zeta _4\left( 7,3-i|\mathbf{a}_4;\boldsymbol{\sigma }_4 \right) 
	+\zeta _4\left( 7,2|\mathbf{a}_4;\boldsymbol{\sigma }_4 \right) -\zeta _4\left( 7,4|\mathbf{a}_4;\boldsymbol{\sigma }_4 \right)\\& =\frac{\Gamma ^{12}}{5242880\pi ^3}-\frac{\Gamma ^{16}}{125829120\pi ^5},
\end{align*}
and
\begin{align*}
&	\zeta _4(5,2|\mathbf{a}_4)-i\zeta _4(5,3-i|\mathbf{a}_4)+i\zeta _4(5,3+i|\mathbf{a}_4)-\zeta _4(5,4|\mathbf{a}_4)
	=\frac{\Gamma ^{12}}{786432\pi ^4}-\frac{\Gamma ^8}{98304\pi ^2},
	\\
&	\zeta _4(9,2|\mathbf{a}_4)-i\zeta _4(9,3-i|\mathbf{a}_4)+i\zeta _4(9,3+i|\mathbf{a}_4)-\zeta _4(9,4|\mathbf{a}_4)
	=\frac{9\Gamma ^{16}}{4697620480\pi ^4}-\frac{11\Gamma ^{20}}{
		112742891520\pi ^6},
	\\
&i\zeta _4(7,3-i|\mathbf{a}_4)-i	\zeta _4(7,3+i|\mathbf{a}_4)+\zeta _4(7,2|\mathbf{a}_4)-\zeta _4(7,4|\mathbf{a}_4)
	=\frac{\Gamma ^{16}}{75497472\pi ^5}-\frac{\Gamma ^{12}}{5242880\pi ^3},
\end{align*}
where $\mathbf{a}_4=\left( 2-2i,2+2i,1-i,1+i \right) $ and $\boldsymbol{\sigma }_4=\left( \left\{ 1 \right\}^2,\left\{ -1 \right\}^2 \right)$.
	\end{exa}

{\bf Conflict of Interest.} The author declare that they have no conflict of interest.

{\bf Data Availability} Data sharing is not applicable to this article as no new data were created or analyzed in this study.

{\bf Use of AI Tools Declaration.} The author declare they have not used Artificial Intelligence (AI) tools in the creation of this article.

{\bf Acknowledgements}: The author would like to thank the referee for his/her careful reading of the paper and many helpful comments.

\end{document}